\documentclass[envcountsame,a4paper,12pt]{elsarticle}
\usepackage[margin=2.3cm,marginpar=1.6cm]{geometry}
\makeatletter
\def\ps@pprintTitle{%
 \let\@oddhead\@empty
 \let\@evenhead\@empty
 \def\@oddfoot{\centerline{\thepage}}%
 \let\@evenfoot\@oddfoot}
\makeatother

\usepackage{microtype}
\usepackage{hyperref}

\usepackage{main,booktabs,multirow}

\begin{document}

\begin{frontmatter}

\title{Characterising Modal Definability of Team-Based Logics via the Universal Modality
}

\author[jap,hok]{Katsuhiko Sano}
\ead{katsuhiko.sano@gmail.com}

\author[jap,hel,han]{Jonni Virtema}
\ead{jonni.virtema@gmail.com}


\address[jap]{Japan Advanced Institute of Science and Technology}

\address[hok]{Hokkaido University}

\address[hel]{University of Helsinki}

\address[han]{Leibniz University Hannover}

\begin{abstract}
We study model and frame definability of various modal logics. Let $\ML(\uBoxp)$ denote the fragment of modal logic extended with the universal modality in which the universal modality occurs only positively. We show that a class of Kripke models is definable in $\ML(\uBoxp)$ if and only if the class is elementary and closed under disjoint unions and surjective bisimulations.
We also characterise the definability of $\ML(\uBoxp)$ in the spirit of the well-known Goldblatt--Thomason theorem. We show that an elementary class $\cF$ of Kripke frames is definable in $\ML(\uBoxp)$ if and only if $\cF$ is closed under taking generated subframes and bounded morphic images, and reflects ultrafilter extensions and finitely generated subframes. In addition we study frame definability relative to finite transitive frames and give an analogous characterisation of $\ML(\uBoxp)$-definability relative to finite transitive frames. Finally, we initiate the study of model and frame definability in team-based logics. We study (extended) modal dependence logic, (extended) modal inclusion logic, and modal team logic. We establish strict linear hierarchies with respect to model definability and frame definability, respectively. We show that, with respect to model and  frame definability, the before mentioned team-based logics, except modal dependence logic, either coincide with $\ML(\uBoxp)$ or plain modal logic $\ML$. Thus as a corollary we obtain model theoretic characterisation of model and frame definability for the team-based logics.

This article subsumes and extends the conference articles \cite{savi15b} and  \cite{savi16}.
\end{abstract}


\end{frontmatter}

\section{Introduction}
Modal logic as a field has progressed far from its philosophical origin, e.g, from the study of the concepts of necessity and possibility. Modern modal logics are integral parts of both theoretical research and real life applications in various scientific fields such as mathematics, artificial intelligence, linguistics, economic game theory, and especially in many subfields of theoretical and applied computer science. Indeed, the general framework of modal logic has been found to be remarkably adaptive.

During the last decade there has been an emergence of vibrant research on logics with \emph{team semantics} in both first-order and modal contexts.
\emph{Team semantics} was introduced by Hodges \cite{Hodges97} in the context of the so-called \emph{independence-friendly logic} of Hintikka and Sandu \cite{hisa89}. The fundamental idea behind team semantics is crisp. The idea is to shift from single assignments to sets of assignments as the satisfying elements of formulas. V\"a\"an\"anen  \cite{va07}  adopted team semantics as the core notion for his \emph{dependence logic}. The syntax of first-order dependence logic extends the syntax of first-order logic by novel atomic formulas called dependence atoms. The intuitive meaning of the dependence atom $\fdep{x_1,\dots,x_n,y}$ is that inside a team the value of $y$ is functionally determined by the values of  $x_1,\dots,x_n$. After the introduction of dependence logic in 2007 the study of related logics with team semantics has bloomed. One of the most important developments in the area of team semantics was the introduction of \emph{independence logic} by Gr\"adel and V\"a\"an\"anen \cite{gradel13} in which dependence atoms of dependence logic are replaced by \emph{independence atoms}. Soon after, Galliani \cite{Galliani:2011} showed that independence atoms can be further analysed, and alternatively expressed, in terms of inclusion and exclusion atoms.

Different dependency notions, such as functional dependence, independence, and inclusion dependence, are important concepts in many areas of science, and especially in statistics and database theory. Using the language of database theory, dependence atoms can be interpreted as \emph{equality generating dependencies} and inclusion dependencies as \emph{tuple generating dependencies}, see e.g., a survey of Kolaitis \cite{kolaitis05} for further information about dependency notions for schema mappings and data exchange. For first works that directly study the connection between dependencies in database theory and dependence logics see the works by Hannula et~al. \cite{hannula15, hks16, hako2016}. Also very recently, a connection between a variant of dependence logic and \emph{constraint satisfaction problems} has been identified by Hella and Kolaitis \cite{hk16}.

Concurrently a vibrant research on modal and propositional logics with team semantics has emerged. In the context of modal logic, any subset of the domain of a Kripke model is called a team. In modal team semantics, formulas are evaluated with respect to team-pointed Kripke models.
The study of \emph{modal dependence logic} was initiated by V\"a\"an\"anen \cite{va08} in 2008. Shortly after, \emph{extended modal dependence logic}  ($\EMDL$) was introduced by Ebbing et~al. \cite{EHMMVV13} and \emph{modal independence logic} by Kontinen et~al. \cite{KMSV14}.
The focus of the research has been in the computational complexity and expressive power. Hella et~al. \cite{HeLuSaVi14} established that exactly the properties of teams that have the so-called \emph{empty team property}, are \emph{downward closed} and closed under the so-called team \emph{$k$-bisimulation}, for some finite $k$, are definable in $\EMDL$. Kontinen et~al. \cite{komu15} have shown that exactly the properties of teams that are closed under the team $k$-bisimulation are definable in the so-called \emph{modal team logic}, whereas Hella and Stumpf established \cite{hest15} that the so-called \emph{extended modal inclusion logic} is characterised by the empty team property, union closure, and closure under team $k$-bisimulation. See the survey \cite{dukovo16} for a detailed exposition on the expressive power and computational complexity of related logics. Whereas the expressive powers of the related logics have been well-studied, the closely related topics of model and frame definability have not been addressed before. Here we mend this shortcoming and all but completely characterise definability of the most studied team-based modal logics.

Modal logic extended with the universal modality ($\ML(\uBox)$) was first formulated by Goranko and Passy \cite{Goranko92}. It extends modal logic by a novel modality $\uBox$, called the universal modality, with the following semantics: the formula $\uBox \varphi$ is true in a point $w$ of a model $\mK$ if $\varphi$ is true in every point $v$ of the model $\mK$. In this article we identify a connection between particular team-based modal logics and a fragment of $\ML(\uBox)$. We will then characterise the fragment of $\ML(\uBox)$  with respect to model and frame definability, and use the connection to team-based logics in order characterise model and frame definability of these team-based modal logics.

The celebrated Goldblatt--Thomason theorem~\cite{GT1975} is a characterisation of modal definability of elementary (i.e., first-order definable) classes of Kripke frames by four frame constructions: generated subframes, disjoint unions, bounded morphic images, and ultrafilter extensions. The theorem states that an elementary class of Kripke frames is definable by a set of modal formulas if and only if the class is closed under taking generated subframes, disjoint unions and bounded morphic images, and reflects ultrafilter extensions. The original proof of Goldblatt and Thomason was algebraic. A model-theoretic version of the proof was later given by van Benthem~\cite{Benthem1993}. From then on, Goldblatt--Thomason -style theorems have been formulated for numerous extensions of modal logic such as modal logic with the universal modality~\cite{Goranko92}, difference logic~\cite{GG1993}, hybrid logic~\cite{tenCate2005}, and graded modal logic~\cite{SM2010}. Also restricted versions of frame definability, such as definability within the class of finite transitive frames  \cite{Benthem1993,GG1993}, have been considered. Also model theoretic characterisations of definable model classes have been given, e.g., for $\ML$ \cite{derijke01}  and $\ML(\uBox)$ \cite{perkov12}. For related work, see also \cite{Perkov2014}.

This paper initiates the study of model and frame definability in the framework of team semantics. Our contribution is two-fold. Firstly, we give Goldblatt--Thomason -style theorems for a fragment of modal logic extended with the universal modality; one restricted to elementary classes and another relative to the class of finite transitive frames. Moreover we give a characterisation of model definability of this logic. Secondly, we show that there is a surprising connection between this fragment and particular team-based modal logics. We also establish surprising strict linear hierarchies with respect to model and frame definability.

Let $\ML(\uBoxp)$ denote the syntactic fragment of $\ML(\uBox)$ in which the universal modality occurs only positively.  We establish that a class $\mathbb{C}$ of Kripke models is definable in $\ML(\uBoxp)$ if and only if $\mathbb{C}$ is closed under surjective bisimulations and ultraproducts, and the complement class $\overline{\mathbb{C}}$ is closed under ultrapowers.
We show that an elementary class of Kripke frames is definable in $\ML(\uBoxp)$ if and only if it is closed under taking generated subframes and bounded morphic images, and reflects ultrafilter extensions and finitely generated subframes.
Moreover we show that a class $\cF$ of finite transitive frames is definable in $\ML(\uBoxp)$ relative to the class of finite transitive frames if and only if $\cF$ is closed under taking generated subframes and bounded morphic images.
Finally we establish that with respect to modal and frame definability a collection of team-based modal logics either coincide with $\ML$ or $\ML(\uBoxp)$. From this connection we obtain characterisations of model and frame definability for each of the related team-based modal logics.
In addition, we obtain strict linear hierarchies for both model and frame definability that include each of the logics studied in this article.

This article is devided in two main parts: Sections \ref{secmlu}--\ref{sec:relmd} concentrate in the study of modal logic with the universal modality and Kripke semantics. In Sections \ref{sec:mlteams}--\ref{sec:mlandum} logics with team-semantics are considered. Moreover in Section \ref{sec:mlandum} these two different formalisms are connected together via model and frame definability.

\section{Modal Logic with Universal Modality}\label{secmlu}

In this section we introduce an extension of the basic modal logic with the universal modality ($\ML(\uBox)$) and present some basic definitions. In addition we present a normal form for the fragment of  $\ML(\uBox)$ in which the universal modality occurs only positively.

\subsection{Syntax and semantics}
In team-based logics it is customary to define the syntax in negation normal form, that is to assume that negations occur only in front of proposition symbols. This is due to the fact that the team semantics negation, that corresponds to the negation used in Kripke semantics, is not the contradictory negation of team semantics. Since in this article we consider extensions of modal logic in the framework of team semantics, we define the syntax of modal logic also in negation normal form.

Let $\Phi$ be a set of atomic propositions. The set of formulas for \emph{modal logic} $\ML(\Phi)$ is generated by the following grammar:
\[
\varphi \ddfn p\mid \neg p \mid (\varphi \wedge \varphi)\mid (\varphi \vee \varphi) \mid \Diamond \varphi \mid \Box \varphi, \quad\text{where $p\in\Phi$.}
\]
The syntax of \emph{modal logic with universal modality} $\ML(\uBox)(\Phi)$ is obtained by extending the syntax of $\ML(\Phi)$ by the grammar rules
\[
\varphi \ddfn \uBox\varphi \mid \uDiamond\varphi.
\]
The syntax of \emph{modal logic with positive universal modality} $\ML(\uBoxp)(\Phi)$ is obtained by extending the syntax of $\ML(\Phi)$ by the grammar rule
\[
\varphi \ddfn \uBox\varphi.
\]
As usual, if the underlying set $\Phi$ of atomic propositions is clear from the context, we drop ``$(\Phi)$'' and just write $\ML$, $\ML(\uBox)$, etc. 
We also use the shorthands $\neg \varphi$, $\varphi\rightarrow\psi$, and $\varphi\leftrightarrow\psi$. By $\neg\varphi$ we denote the formula that can be obtained from $\neg\varphi$ by pushing all negations to the atomic level, and by $\varphi\rightarrow\psi$ and $\varphi\leftrightarrow\psi$,  we denote $(\neg\varphi\lor\psi)$ and $(\varphi\rightarrow\psi)\land(\psi\rightarrow\varphi)$, respectively.

A (Kripke) {\em frame} is a pair $\fF=(W,R)$ where $W$, called the {\em domain} of $\fF$, is a non-empty set and $R\subseteq W\times W$ is a binary relation on $W$. By $\frames$, we denote the class of all frames.  We use $|\fF|$ to denote the domain of the frame $\fF$. Let $\Phi$ be a set of proposition symbols. A (Kripke) {\em $\Phi$-model} is a tuple $\mK=(W,R,V)$, where $(W,R)$ is a frame and $V:\Phi\to \mathcal{P}(W)$ is a valuation of the proposition symbols.  By $\modls(\Phi)$, we denote the class of all $\Phi$-models.
The semantics of modal logic, i.e., the \emph{satisfaction relation} $\mK, w \Vdash \varphi$, is defined via pointed \emph{$\Phi$-models} as usual, see, e.g., \cite{blackburn:2001}. For the universal modality $\uBox$ and its dual $\uDiamond$, we define
\[
\begin{array}{lll}
\mK, w \Vdash \uBox\varphi&\quad \Leftrightarrow \quad & \mK, v \Vdash \varphi, \text{ for every }v\in W,\\
\mK, w \Vdash \uDiamond\varphi &\quad \Leftrightarrow \quad & \mK, v \Vdash \varphi, \text{ for some }v\in W.\\
\end{array}
\]
 We say that formulas $\varphi$ and $\psi$ \emph{are equivalent in Kripke semantics} ($\varphi\equiv_{K}\psi$), if the equivalence $\mK,w \Vdash \varphi \Leftrightarrow\mK,w\Vdash\psi$ holds for every model $\mK=(W,R,V)$ and every $w\in W$. 
If $\varphi\in\ML(\uBox)(\Phi)$ is a Boolean combination of formulas beginning with $\uBox$, we say that $\varphi$ is \emph{closed}.

A formula set $\Gamma$ is \emph{valid in a model} $\mK=(W,R,V)$ (notation: $\mK\Vdash \Gamma$), if $\mK,w\Vdash \varphi$ holds for every $w\in W$ and every $\varphi \in \Gamma$. The set $\Gamma$ is \emph{valid in a class $\cC$ of models}  (written: $\cC\Vdash \Gamma$) if  $\mK\Vdash \Gamma$ for every $\mK\in\cC$. When $\Gamma$ is a singleton $\setof{\varphi}$, we simply write $\mK\Vdash \varphi$ and $\cC\Vdash\varphi$.
Similarly, a formula set $\Gamma$ is \emph{valid in a frame} $\fF=(W,R)$ (notation: $\fF\Vdash \Gamma$), if $\Gamma$ is valid in every model of the form $(\fF,V)$. A set $\Gamma$ of $\LL$-formulas is {\em valid in a class $\cF$ of frames}  (written: $\cF\Vdash \Gamma$) if $\fF \Vdash \Gamma$ for every $\fF \in \cF$. Again when $\Gamma$ is a singleton $\setof{\varphi}$, we simply write $\fF\Vdash \varphi$ and $\cF\Vdash \varphi$.

\subsection{Definability}\label{sec:definability}
Let $\LL(\Phi)$ and $\LL'(\Phi')$ be modal logics such that the validity relation for Kripke models (i.e., $\mK\Vdash \Gamma$) is defined. Given a set $\Gamma$ of $\LL(\Phi)$-formulas, we define 
\[
\modlss(\Gamma) := \inset{\mK \in \modls(\Phi)}{\mK \Vdash \Gamma} \text{ and } \framess(\Gamma) := \inset{\mathfrak{F} \in \frames}{\mathfrak{F} \Vdash \Gamma}. 
\]
We say that $\Gamma$ \emph{defines} a class $\cC$ of models (frames), if $\cC=\modlss(\Gamma)$ ($\cC=\framess(\Gamma)$). When $\Gamma$ is a singleton $\setof{\varphi}$, we simply say that $\varphi$ defines $\cC$.
A class $\cC$ of models (frames) is \emph{$\LL(\Phi)$-definable} if there exists a set $\Gamma$ of $\LL(\Phi)$-formulas such that $\modlss(\Gamma)=\cC$ ($\framess(\Gamma)=\cC$).

We write $\LL(\Phi)\mleq\LL'(\Phi)$, if every $\LL(\Phi)$-definable class of models is also $\LL'(\Phi)$-definable. We write $\LL(\Phi)\meq\LL'(\Phi)$, if both $\LL(\Phi)\mleq\LL'(\Phi)$ and $\LL'(\Phi)\mleq\LL(\Phi)$ hold, and write $\LL(\Phi) \mle \LL'(\Phi)$, if $\LL(\Phi)\mleq\LL'(\Phi)$ but $\LL'(\Phi)\not\mleq\LL(\Phi)$. Analogously, we write $\LL(\Phi)\fleq\LL'(\Phi')$, if every $\LL(\Phi)$-definable class of frames is also $\LL'(\Phi')$-definable. We write $\LL(\Phi)\feq\LL'(\Phi')$, if both $\LL(\Phi)\fleq\LL'(\Phi')$ and $\LL'(\Phi')\fleq\LL(\Phi)$ hold, and write $\LL(\Phi) \fle \LL'(\Phi')$ if $\LL(\Phi)\fleq\LL'(\Phi')$ but $\LL'(\Phi')\not\fleq\LL(\Phi)$. 

A class $\cC$ of $\Phi$-models is called \emph{elementary} if there exists a set of first-order sentences with equality of the vocabulary  $\Phi\cup\{R \}$ that defines $\cC$. A class $\fF$ of frames is called \emph{elementary} if there exists a set of first-order sentences with equality of the vocabulary  $\{R \}$ that defines $\cC$. 

It is well-known that, via the so-called \emph{standard translation}, formulas of modal logic can be translated to formulas of first-order logic with one free variable. This translation also geneneralises to $\ML(\uBox)$. Thus it follows that every  $\ML(\uBox)$-definable class of models is elementary. However this does not hold for classes of frames for the obvious reason; in frame definability the univeral quantification of valuations corresponds to quantification of sets and thus the corresponding translation is to monadic second-order logic. 
\begin{definition}[Standard translation]\label{def:st}
Let $x$ be a first-order variable. The \emph{standard translation} $\ST_x$ that maps formulas of $\ML(\uBox)(\Phi)$ to formulas of $\FO(\{R\}\cup \Phi)$ is defined as follows:
\begin{align*}
\ST_x(p) \,=\,& P(x),\\ 
\ST_x(\neg p) \,=\,& \neg P(x),\\ 
\ST_x(\varphi \lor \psi) \,=\,& \ST_x(\varphi) \lor  \ST_x(\psi), \\
\ST_x(\varphi \land \psi) \,=\,& \ST_x(\varphi) \land  \ST_x(\psi), \\
\ST_x(\Diamond\varphi) \,=\,& \exists y \big(R(x,y) \land \ST_y(\varphi)\big), \\
\ST_x(\Box\varphi) \,=\,& \forall y \big(R(x,y) \rightarrow \ST_y(\varphi)\big), \\
\ST_x(\uDiamond\varphi) \,=\,& \exists x \ST_x(\varphi), \\ 
\ST_x(\uBox\varphi) \,=\,& \forall x \ST_x(\varphi),
\end{align*}
where $y$ is a fresh variable. 
\end{definition}
The proof of the following proposition is self-evident; for basic modal logic $\ML$ see \cite[Proposition 2.47]{blackburn:2001}.
\begin{proposition} \label{prop:st}
Let $\varphi$ be an $\ML(\uBox)$-formula.
\begin{enumerate}
\item For every $\mK$ and every point $w$ of $\mK$: $\mK,w \Vdash \varphi$ iff  $\mK \models_{\FO} \ST_x(\varphi)[w]$.
\item For every $\mK$: $\mK \Vdash \varphi$ iff  $\mK \models_{\FO} \forall x \ST_x(\varphi)$.
\end{enumerate}
Here $\models_{\FO}$ denotes the satisfaction relation of first-order logic.
\end{proposition}

\subsection{Normal Form}
We will next define a normal form for $\ML(\uBoxp)$. This normal form is a modification of the normal form for $\ML(\uBox)$ by Goranko and Passy in \cite{Goranko92}.

\begin{definition} 
%
\begin{enumerate}
\item[$(\mathrm{i})$] A formula $\varphi$ is a \emph{disjunctive $\uBox$-clause} if there exists a natural number $n\in\omega$ and formulas $\psi, \psi_1,\dots,\psi_n\in \ML$ such that $\varphi$ $=$ $\psi \vee \uBox\psi_1\vee \dots \vee \uBox\psi_n$. 
\item[$(\mathrm{ii})$] A formula $\varphi$ is in \emph{conjunctive $\uBox$-form}  if $\varphi$ is a conjunction of disjunctive $\uBox$-clauses.
\item[$(\mathrm{i})$] A formula $\varphi$ is a \emph{conjunctive $\uBox$-clause} if there exists formulas $\psi, \theta\in \ML$ such that $\varphi = \psi \land \uBox\theta$. 
\item[$(\mathrm{ii})$] A formula $\varphi$ is in \emph{disjunctive $\uBox$-form} if $\varphi$ is a disjunction of conjunctive $\uBox$-clauses.
\item[$(\mathrm{iii})$] A formula $\varphi$ is in \emph{$\uBox$-form} if $\varphi$ is either in conjunctive $\uBox$-form or in disjunctive $\uBox$-form.
\end{enumerate}
\end{definition}
It is easy to show that for each $\ML(\uBoxp)$-formula in conjunctive $\uBox$-form there exists an equivalent $\ML(\uBoxp)$-formula in disjunctive $\uBox$-form, and vice versa.

Recall that $\varphi\in\ML(\uBox)(\Phi)$ is {\em closed} if it is a Boolean combination of formulae beginning with $\uBox$. 
\begin{proposition}
\label{equivalences}
Let $\varphi, \psi \in \ML(\uBox)$ such that $\psi$ is closed. Then, 
\begin{multicols}{2}
\begin{enumerate}
\item $\Box(\varphi \vee \psi) \equiv_{K} (\Box\varphi \vee \psi)$,\label{ecase1} 
\item $\Diamond(\varphi \wedge \psi) \equiv_{K} (\Diamond\varphi \wedge \psi)$,\label{ecase2}
\item $\uBox(\varphi \vee \psi) \equiv_{K} (\uBox\varphi \vee \psi)$.\label{ecase3}
\end{enumerate}
\end{multicols}
\end{proposition}
\begin{proof}
Cases \ref{ecase1} and \ref{ecase3} follow from \cite[Proposition 3.6]{Goranko92}. Case \ref{ecase2} is completely analogous to case \ref{ecase1}.
\end{proof}

\begin{theorem}
\label{A+form}
For each $\ML(\uBoxp)$-formula $\varphi$, there exists an $\ML(\uBoxp)$-formula $\psi$ in $\uBox$-form such that $\varphi\equiv_K\psi$. 
\end{theorem}
\begin{proof}
The proof is done by induction on $\varphi$. The cases for literals and connectives are trivial. 
As for the case $\varphi = \Box \psi$, we proceed as follows. By induction hypothesis there exists a conjunctive $\uBox$-form $\bigwedge_{i \in I} \psi_{i}$, where each $\psi_i$ is a disjunctive $\uBox$-clause, such that $\bigwedge_{i \in I} \psi_{i} \equiv_{K}\psi$. By the semantics of $\Box$, we then have that
\[
\Box\psi \equiv_{K} \Box \bigwedge_{i \in I} \psi_{i} \equiv_{K} \bigwedge_{i \in I} \Box \psi_{i}.
\]
Now since each $\psi_i$ is a disjunctive $\uBox$-clause, it follows from case \ref{ecase1} of Proposition \ref{equivalences} that, for each $i \in I$, the formula $\Box\psi_i$ is equivalent to some disjunctive $\uBox$-clause $\psi_i'$. Thus $\bigwedge_{i \in I} \psi_{i}'$ is a conjunctive $\uBox$-form that is equivalent to $\Box\psi$. 

The proof for the case of $\uBox\varphi$ is otherwise the same as the proof for the case $\Box\varphi$, but instead of item \ref{ecase1} of Proposition \ref{equivalences}, item \ref{ecase3} is used. The proof for the case $\Diamond\varphi$ is likewise analogous to that of $\Box\varphi$. The proof uses a disjunctive $\uBox$-form instead of the conjunctive one and item \ref{ecase2} of Proposition \ref{equivalences} instead of item \ref{ecase1}.
\end{proof}

\section{Definability in Modal Logics with Universal modality}
In this section we characterise the definability of the logics introduced in Section \ref{secmlu} with respect to (non-pointed) models. We start by introducing the well-known concepts; disjoint unions and bisimulations.

\begin{definition}[Disjoint Union]
Let $\inset{\mK_{i}}{i \in I}$ be a pairwise disjoint family of $\Phi$-models, where $\mK_{i}$ = $\tuple{W_{i},R_{i}, V_i}$. The {\em disjoint union} $\biguplus_{i \in I} \mK_{i}$ = $\tuple{W,R,V}$ of $\inset{\mK_{i}}{i \in I}$ is defined by $W = \bigcup_{i \in I} W_{i}$, $R = \bigcup_{i \in I} R_{i}$, and $V(p)= \bigcup_{i \in I} V_{i}(p)$, for each $p\in\Phi$. 
\end{definition}

\begin{definition}[Bisimulation]\label{def:bisim}
Let $\mK=(W,R,V)$ and $\mK'=(W',R',V')$ be $\Phi$-models.
A nonempty relation $Z  \subseteq{W\times W'}$ is called a \emph{bisimulation} if for each $(w,w')\in Z$ it holds that
\begin{enumerate}
\item $\mK,w \Vdash p \Leftrightarrow \mK',w' \Vdash p$, for each $p\in\Phi$,
\item for each $v\in W$ s.t. $wRv$ there exists $v'\in W'$ s.t. $w'R'v'$ and $vZv'$,
\item for each $v'\in W'$ s.t. $w'R'v'$ there exists $v\in W$ s.t. $wRv$ and $vZv'$.
\end{enumerate}
If the domain of $Z$ is $W$, we call $Z$ total, and if the range of $Z$ is $W$', we say that $Z$ is surjective. 
\end{definition}

It is well-known that for pointed models and basic modal logic $\ML$ bisimulation implies modal equivalence. Moreover with respect to modal definability, we have the following characterisation. Ultraproducts and ultrapowers are standard notions of first-order model theory, see e.g., the book of Chang and Keisler \cite{CK1990}. In this paper these notions are used in order to build $\omega$-saturated and elementary equivalent models (again standard notions of first-order model theory, see e.g., Chang and Keisler)  from given Kripke models.
\begin{theorem}[\citep{derijke01,perkov12}]\label{MLmodels}
Let $\cC$ be a class of Kripke models. The following equivalences hold:
\begin{enumerate}
\item The class $\cC$ is definable in $\ML$ if and only if $\cC$ is closed under surjective bisimulations, disjoint unions and ultraproducts, and $\overline{\cC}$ is closed under ultrapowers.
\item The class $\cC$ is definable in $\ML(\uBox)$ if and only if $\cC$ is closed under total surjective bisimulations and ultraproducts, and $\overline{\cC}$ is closed under ultrapowers.
\end{enumerate}
\end{theorem}
It is well-known (see, e.g., \cite{CK1990}) that a class of models $\cC$ is elementary if and only if it is closed under isomorphisms and ultraproducts, while its complement is closed under ultrapowers. Thus the above theorem may be rewritten as follows:
\begin{corollary}\label{cor:oldelementary}
Let $\cC$ be a class of Kripke models. The following equivalences hold:
\begin{enumerate}
\item The class $\cC$ is definable in $\ML$ if and only if $\cC$ is elementary and closed under surjective bisimulations and disjoint unions.
\item The class $\cC$ is definable in $\ML(\uBox)$ if and only if $\cC$ is elementary and closed under total surjective bisimulations.
\end{enumerate}
\end{corollary}

We will next establish corresponding characterisations for $\ML(\uBoxp)$.
Recall that a closed disjunctive $\uBox$-clause is a formula of the form $\bigvee_{i \in I} \uBox \varphi_{i}$, where, for each $i\in I$, $\varphi_{i} \in \ML$. 
\begin{definition}
By $\bigvee \uBox\ML$ we denote the set of all closed disjunctive $\uBox$-clauses.
\end{definition}

\begin{lemma}
\label{cor:A+form}
For each $\ML(\uBoxp)$-formula $\varphi$, there exists a finite set $\Gamma$ of closed disjunctive $\uBox$-clauses such that 
$\mathfrak{M} \Vdash \varphi$ iff $\mathfrak{M} \Vdash \Gamma$, for every model $\mathfrak{M}$. 
\end{lemma}
\begin{proof}
Let $\varphi$ be an $\ML(\uBoxp)$-formula. 
By Theorem \ref{A+form}, we may assume that $\varphi$ is a conjunctive $\uBox$-form $\bigwedge_{i \in I} \psi_{i}$, where each $\psi_{i}$ := $\gamma_{i} \lor \bigvee_{ j \in J_{i}} \uBox \delta_{j}$ is a disjunctive $\uBox$-clause.
By item \ref{ecase3} of Proposition \ref{equivalences}, for each $i\in I$, $\uBox \psi_{i}$ is equivalent to the closed disjunctive  $\uBox$-clause $\psi_{i}'$ := $\uBox \gamma_{i} \lor \bigvee_{ j \in J_{i}} \uBox \delta_{j}$.
Thus, for every model $\mK$,
\[
\mK\Vdash {\bigwedge}_{i \in I} \psi_{i} \Leftrightarrow \mK\Vdash \{\psi_i \mid i\in I\} \Leftrightarrow \mK\Vdash \{\uBox\psi_i \mid i\in I\} \Leftrightarrow \mK\Vdash \{\psi_i' \mid i\in I\}.
\]
\end{proof}

\begin{proposition}
\label{prop:mlup_eq_vuml_m}
A class $\cC$ of Kripke models is definable in
$\ML(\uBoxp)$ if only if $\cC$ is definable in $\bigvee \uBox\ML$. 
\end{proposition}
\begin{proof}
The direction $\bigvee \uBox\ML$ $\mleq$ $\ML(\uBoxp)$ is trivial. We will establish that $\ML(\uBoxp)\mleq\bigvee \uBox\ML$. 
Consider any $\ML(\uBoxp)$-definable class of models $\mathbb{C}$. Let $\Gamma$ be a set of $\ML(\uBoxp)$ formulas that defines $\mathbb{C}$. 
By Lemma \ref{cor:A+form}, for each $\varphi \in \Gamma$, there is a finite set $\Delta_{\varphi}$ of closed disjunctive $\uBox$-clauses such that $\mathfrak{M} \Vdash \varphi$ iff $\mathfrak{M} \Vdash \Delta_{\varphi}$, for every Kripke model $\mathfrak{M}$. It follows that $\mathfrak{M} \Vdash \Gamma$ iff $\mathfrak{M} \Vdash \bigcup_{\varphi \in \Gamma} \Delta_{\varphi}$, for every Kripke model $\mathfrak{M}$. Therefore $\bigcup_{\varphi \in \Gamma} \Delta_{\varphi}$ defines $\mathbb{C}$ as desired. 
\end{proof}

\begin{proposition}\label{prop:closedsurjective}
Let $\mK$ and $\mK'$ be Kripke models such that there is a surjective bisimulation from $\mK$ to $\mK'$, and let $\varphi$ be an $\ML(\uBoxp)$-formula. If $\mK\Vdash \varphi$ then $\mK'\Vdash \varphi$.
\end{proposition}
\begin{proof}
Let $Z\subseteq W\times W'$ be a surjective bisimulation. We show by structural induction that for every $\varphi\in\ML(\uBoxp)$ and $(w,w')\in Z$, it holds that, if $\mK,w\Vdash \varphi$ then $\mK',w'\Vdash \varphi$, from which the claim follows. The cases for (negated) propositional symbols, Boolean connectives, and the modalities $\Diamond$ and $\Box$ are standard. We show the case for $\uBox$. Assume that $\mK,w\Vdash \uBox\varphi$. Now, for every $v\in \dom{Z}$, it holds that $\mK,v\Vdash \varphi$. Thus, by induction hypothesis, $\mK',v'\Vdash \varphi$, for every $v'\in \ran{Z}=W'$. Thus $\mK',w'\Vdash \uBox\varphi$.
\end{proof}

\begin{lemma}[\cite{perkov12}]\label{lemma:surjection}
Let $\mK$ and $\mK'$ be $\omega$-saturated Kripke models. Assume that, for every $\varphi\in\ML$, $\mK\Vdash\varphi$ implies $\mK'\Vdash\varphi$. Then there exists a surjective bisimulation from $\mK$ to $\mK'$.
\end{lemma}

The following result was essentially stated in the conclusion of \cite{derijke01}. We present a detailed proof.
\begin{theorem}
A class $\mathbb{C}$ of models is definable in $\bigvee \uBox\ML$ if and only if $\mathbb{C}$ is closed under surjective bisimulations and ultraproducts, and $\overline{\mathbb{C}}$ is closed under ultrapowers.
\end{theorem}

\begin{proof}
Let $\mathbb{C}$ be a class of models definable in $\bigvee \uBox\ML$. Via standard translation (Proposition \ref{prop:st}), we get that $\mathbb{C}$ is elementary and thus $\mathbb{C}$ is closed under ultraproducts and $\overline{\mathbb{C}}$ is closed under ultrapowers. By Propositions \ref{prop:mlup_eq_vuml_m} and \ref{prop:closedsurjective}, $\mathbb{C}$ is closed under surjective bisimulations.

Assume that $\mathbb{C}$ is closed under surjective bisimulations and ultraproducts, and that $\overline{\mathbb{C}}$ is closed under ultrapowers. Define
\[
S := \{ \varphi \in \bigvee \uBox\ML \mid   \mathbb{C} \Vdash \varphi   \}
\]
We will show that $S$ defines $\mathbb{C}$. Clearly $\mathbb{C}\subseteq\modlss(S)$, thus we show the converse. Let $\mK$ be a model such that $\mK\Vdash S$. Let $U$ be some ultrafilter over $\mathbb{N}$ that includes all cofinite subsets of $\mathbb{N}$ (for the existence of such ultrafilter, the reader is referred to, e.g., ~\cite[Proposition 3.3.6]{CK1990}) and put $\mK'=\Pi_{U}\mK$. Clearly $U$ is a countably incomplete ultrafilter over $\mathbb{N}$ (an ultrafilter is countably incomplete if it is not closed under countable intersections). Note that $\mK$ and $\mK'$ are elementary equivalent (see, e.g., \cite[Corollary A.20]{blackburn:2001}), and thus via standard translation (Proposition \ref{prop:st}) $\mK'\Vdash S$. Define
\[
\Delta := \{ \varphi \in \bigvee \uBox\ML \mid   \mK' \Vdash \varphi   \}.
\]
Define $\overline{\Delta}$ := $\bigvee\uBox\ML \setminus \Delta$. For each $\gamma \in \overline{\Delta}$, $\gamma$ is falsifiable in $\mathbb{C}$, i.e., there exists a model $\mK \in \mathbb{C}$ such that $\mK \not\Vdash \gamma$. 
For if not, $\gamma \in \overline{\Delta}$ is in $S$ (and thus in $\Delta$), a contradiction.
Since the logics we consider have only countably many formulas, we may write $\overline{\Delta} = \{\gamma_1, \dots, \gamma_k, \dots  \}$. Since $\overline{\Delta}$ is closed under disjunctions, for each $k\in\mathbb{N}$, $\gamma_1\lor \cdots \lor \gamma_k$ is falsifiable in $\mathbb{C}$. So, for each $k \in \mathbb{N}$, let us fix a model $\mK'_k \in \mathbb{C}$ such that $\mK'_k \not\Vdash \gamma_1\lor \cdots \lor \gamma_k$. Note that for every formula $\varphi \in \bigvee \uBox\ML$ and Kripke model $\mN$, $\mN\not\Vdash \varphi$ iff  $\mN\Vdash \neg\varphi$. Thus $\mK'_k \Vdash \neg (\gamma_1\lor \cdots \lor \gamma_k)$.  Recall that $U$ contains all cofinite subsets of $\mathbb{N}$. It is easy to check that, by \L{}o\'s's theorem (see, e.g., \cite[Theorem A.19]{blackburn:2001}), $\Pi_U \mK'_k\Vdash \neg \gamma$, for each $\gamma \in \overline{\Delta}$. Thus $\Pi_U \mK'_k\not\Vdash \gamma$, for each $\gamma \in \overline{\Delta}$. 

Recall that $U$ is countably incomplete. By \cite[Theorem 6.1.1]{CK1990} it follows that $\Pi_U \mK'_k$ and $\mK'=\Pi_{U}\mK$ are $\omega$-saturated.
Next we show that, for every $\psi\in\ML$, if $\Pi_U \mK'_k \Vdash \psi$ then $\mK' \Vdash \psi$. By Lemma \ref{lemma:surjection} it then follows that there exists a surjective bisimulation from $\Pi_U \mK'_k$ to $\mK'$.  Assume that $\mK' \not\Vdash \psi$. Thus $\mK' \not\Vdash \uBox \psi$. Now clearly $\uBox \psi \in \overline{\Delta}$, and thus $\Pi_U \mK'_k\not\Vdash \uBox \psi$. It then follows that $\Pi_U \mK'_k \not\Vdash \psi$.

We are now ready to finalise the proof. Recall that, for each $k\in\mathbb{N}$, the model $\mK'_k\in \mathbb{C}$. By assumption $\mathbb{C}$ is closed under ultraproducts and thus $\Pi_U \mK'_k \in \mathbb{C}$. Now since $\mathbb{C}$ is closed under surjective bisimulations, we obtain that $\mK'\in\mathbb{C}$. Recall that $\mK'=\Pi_{U}\mK$ and that $\overline{\mathbb{C}}$ is closed under ultrapowers, thus we conclude that $\mK\in\mathbb{C}$.
\end{proof}

Now together with Proposition \ref{prop:mlup_eq_vuml_m}, we obtain the following corollary.
\begin{corollary}\label{MLupmodels}
A class $\mathbb{C}$ of models is definable in $\ML(\uBoxp)$ if and only if $\mathbb{C}$ is closed under surjective bisimulations and ultraproducts, and $\overline{\mathbb{C}}$ is closed under ultrapowers.
\end{corollary}
Analogously to Corollary \ref{cor:oldelementary}, we obtain the following:
\begin{corollary}\label{MLupmodelsb}
A class $\mathbb{C}$ of models is definable in $\ML(\uBoxp)$ if and only if $\mathbb{C}$ is elementary and closed under surjective bisimulations.
\end{corollary}
With the help of the characterisations above, it is easy to show that the following strict hierarchy follows.
\begin{proposition}\label{ummorder}
$\ML\mle\ML(\uBoxp)\mle\ML(\uBox)$.
\end{proposition}
\begin{proof}
$\ML\mle\ML(\uBoxp)$: Let $\cC$ be the class of $\{p\}$-models that is defined by the $\ML(\uBoxp)$-formula $\uBox p \lor \uBox \neg p$. It is self-evident that $\cC$ is not closed under disjoint unions. Thus, by \Cref{MLmodels}, $\cC$ is not $\ML$-definable.

$\ML(\uBoxp)\mle\ML(\uBox)$:  Let $\cC$ be the class of $\{p\}$-models that is defined by the $\ML(\uBox)$-formula $\uDiamond p$. It is self-evident that $\cC$ is not closed under surjective bisimulations. Thus, by \Cref{MLupmodels}, $\cC$ is not $\ML(\uBoxp)$-definable.
\end{proof}

\section{Modal Frame Definability}\label{secmlframe}
%


In this section we compare $\ML$, $\ML(\uBoxp)$, and $\ML(\uBox)$ with respect to frame definability.
It is easy to see that $\ML \fleq \ML(\uBoxp) \fleq \ML(\uBox)$. To show that the two occurrences of $\fleq$ here are strict, let us introduce two frame constructions.

\begin{definition}[Disjoint Unions]
Let $\inset{\fF_{i}}{i \in I}$ be a pairwise disjoint family of frames, where $\fF_{i}$ = $\tuple{W_{i},R_{i}}$. The {\em disjoint union} $\biguplus_{i \in I} \fF_{i}$ = $\tuple{W,R}$ of $\inset{\fF_{i}}{i \in I}$ is defined by $W = \bigcup_{i \in I} W_{i}$ and $R = \bigcup_{i \in I} R_{i}$. 
\end{definition}

\begin{definition}[Generated Subframes]
Given any two frames $\fF = \tuple{W,R}$ and $\fF = \tuple{W',R'}$, $\fF'$ is a {\em{generated subframe}} of $\fF$ if $(\mathrm{i})$ $W' \subseteq W$, $(\mathrm{ii})$ $R'$ = $R \cap (W')^{2}$, $(\mathrm{iii})$  $w'Rv'$ implies $v'\in W'$, for every $w' \in W'$. We say that $\fF'$ is the {\em{generated subframe}} of $\fF$ by $X \subseteq |\fF|$ (notation: $\fF_{X}$) if $\fF'$ is the smallest generated subframe of $\fF$ whose domain contains $X$. $\fF'$ is a {\em finitely generated subframe} of $\fF$ if there is a finite set $X \subseteq |\fF|$ such that $\fF'$ is $\fF_{X}$. 
\end{definition}

It is well-known that every $\ML$-definable frame class is closed under taking both disjoint unions and generated subframes (see~\cite[Theorem 3.14 (i), (ii)]{blackburn:2001}). However this is not the case for every $\ML(\uBox)$-definable nor every $\ML(\uBoxp)$-definable class; see the following example.
\begin{example}\label{example1}
Consider the following examples from~\cite[p.14]{Goranko92}: the formula $\neg p \lor \uBox p$ defines the class $\{(W,R)\in \frames\mid \lvert W \rvert = 1\}$, whereas the formula $\uDiamond \Diamond (p\lor\neg p)$ defines the class $\{(W,R)\in \frames\mid R \neq \emptyset\}$. Clearly, the former is not closed under taking disjoint unions, and the latter is not closed under taking generated subframes. Note that both of the classes above are elementary.
\end{example}
The above example shows that there exists an $\ML(\uBoxp)$-definable class that is not closed under taking disjoint unions and that there exists an $\ML(\uBox)$-definable class that is not closed under taking generated subframes. Thus we obtain that  $\ML \fle \ML(\uBoxp)$. Next we will next establish that every $\ML(\uBoxp)$-definable frame class is closed under taking generated subframes. From this we get that $\ML(\uBoxp) \fle \ML(\uBox)$.

First note that the following result follows directly from Proposition \ref{prop:mlup_eq_vuml_m}.
\begin{proposition}
\label{prop:mlup_eq_vuml}
A class $\cF$ of Kripke frames is definable in
$\ML(\uBoxp)$ if only if it is definable in $\bigvee \uBox\ML$. 
\end{proposition}
%

\begin{proposition}
\label{prop:gen_frames}
Let $\fF$ be a frame and $\varphi\in \bigvee \uBox\ML$. 
If $\fF \Vdash \varphi$, then $\fG \Vdash \varphi$ for all generated subframes $\fG$ of $\fF$.
\end{proposition}
\begin{proof}
Fix any generated subframe $\fG$ of a frame $\fF$ and put $\varphi$ := $\bigvee_{i \in I} \uBox \psi_{i}$.
Suppose that $\fF \Vdash \varphi$. 
To show $\fG \Vdash \varphi$, fix any valuation $V$ and any state $w$ in $\fG$. 
We show that $(\fG,V),w \Vdash \uBox \psi_{i}$ for some $i \in I$. 
Since we can regard $V$ as a valuation on $\fF$, $(\fF,V), w \Vdash \bigvee_{i \in I} \uBox \psi_{i}$. Thus there is some $i \in I$ such that 
$(\fF,V), u \Vdash \psi_{i}$, for every $u\in|\fF|$. 
Fix such $i \in I$. 
Since $\psi_{i}$ is in $\ML$ and the satisfaction of $\ML$ is invariant under taking generated submodels (cf.~\cite[Proposition 2.6]{blackburn:2001}), 
$(\fG,V), u \Vdash \psi_{i}$ for every $u\in\fG$. 
Therefore, $(\fG,V), w \Vdash \uBox \psi_{i}$, as desired.
\end{proof}

The following proposition follows directly by Propositions \ref{prop:mlup_eq_vuml} and \ref{prop:gen_frames}.
%
\begin{proposition}\label{uboxgen}
Every $\ML(\uBoxp)$-definable frame class is closed under taking generated subframes.
\end{proposition}
Now recall that, by Example \ref{example1}, $\ML(\uBoxp)$ is not closed under taking disjoint unions and $\ML(\uBox)$ is not closed under generated submodels. Furthermore recall that, by Proposition \ref{prop:mlup_eq_vuml}, $\ML(\uBoxp) \feq \bigvee \uBox\ML$. The following strict hierarchy follows.
\begin{proposition}
\label{prop:frame_def_ordering}
$\ML \fle \ML(\uBoxp) \feq \bigvee \uBox\ML \fle \ML(\uBox)$. 
Moreover, the same holds when we restrict ourselves to elementary frame classes.
\end{proposition}
\section{\texorpdfstring{Goldblatt--Thomason -style Theorem for $\ML(\uBoxp)$}{Goldblatt--Thomason -style Theorem for modal logic with positive universal modality}} \label{secgbth}
%


In addition to disjoint unions and generated subframes, we introduce two more frame constructions. With the help of these four constructions, we first review the existing characterisations of $\ML$- and $\ML(\uBox)$-definability when restricted to the elementary frame classes. We then give a novel characterisation of $\ML(\uBoxp)$-definability again restricted to the elementary frame classes. 

\begin{definition}[Bounded Morphism]
Given any two frames $\fF$ = $(W,R)$ and 
$\fF'$ = $(W',R')$, a function $f: W \to W'$ is a {\em{bounded morphism}} if it satisfies the following two conditions:
\begin{description}
\item[$\mathbf{(Forth)}$] If $wRv$, then $f(w)R'f(v)$. 
\item[$\mathbf{(Back)}$] If $f(w)R'v'$, then $wRv$ and $f(v)$ = $v'$ for some $v \in W$. 
\end{description}
If $f$ is surjective, we say that $\fF'$ is a {\em bounded morphic image} of $\fF$. 
\end{definition}

\begin{definition}[Ultrafilter Extensions]
\label{dfn:ultra_filter_ex}
Let $\fF = \tuple{W,R}$ be a Kripke frame, and $\mathrm{Uf}(W)$ denote the set of all ultrafilters on $W$.
Define the binary relation $R^{\mathfrak{ue}}$ on the set $\mathrm{Uf}(W)$ as follows: 
$\mathcal{U} R^{\mathfrak{ue}} \mathcal{U'}$ iff $X \in \mathcal{U}'$ implies $m_{R}(X) \in \mathcal{U}$, for every $X \subseteq W$, where $m_{R}(X)$ $:=$ $\inset{w \in W}{\text{$wRw'$ for some $w' \in X$}}$. 
The frame $\mathfrak{ueF}=(\mathrm{Uf}(W),R^{\mathfrak{ue}})$
is called the {\em ultrafilter extension} of $\fF$.
\end{definition}

A frame class $\cF$ {\em reflects} ultrafilter extensions if $\mathfrak{ueF} \in \cF$ implies $\fF \in \cF$ for every frame $\fF$. It is well-known that every $\ML$- or $\ML(\uBox)$-definable frame class is closed under taking bounded morphic images and reflects ultrafilter extensions (cf.~\cite[Theorem 3.14, Corollary 3.16 and Exercise 7.1.2]{blackburn:2001}). 

\begin{theorem}[Goldblatt--Thomason theorems for $\ML$~\cite{GT1975} and $\ML(\uBox)$~\cite{Goranko92}]\label{gbth}
$\mathrm{(i)}$\; An elementary frame class is $\ML$-definable if and only if it is closed under taking bounded morphic images, generated subframes, disjoint unions and reflects ultrafilter extensions.

\noindent $\mathrm{(ii)}$\; An elementary frame class is $\ML(\uBox)$-definable if and only if it is closed under taking bounded morphic images and reflects ultrafilter extensions.
\end{theorem}

In order to characterise $\ML(\uBoxp)$-definability of elementary frame classes, we need to introduce the following notion of {\em reflection of finitely generated subframes}:
a frame class $\cF$ {\em reflects} finitely generated subframes whenever it is the case for all frames $\fF$ that, if every finitely generated subframe of $\fF$ is in $\cF$, then $\fF \in \cF$.\footnote{Closure under generated subframes and reflection of finitely generated subframes characterise the definability of hybrid logic with satisfaction operators and downarrow binder when restricted elementary frame classes~\cite[Theorem 26]{AC2007}. } In what follows, we show that every $\ML(\uBoxp)$-definable class of Kripke frames reflects finitely generated subframes via the following intermediate result for $\bigvee \uBox\ML$. 

\begin{proposition}
\label{prop:fin_gen_frames}
Let $\fF$ be a frame and $\varphi\in\bigvee \uBox\ML$. 
If $\fG \Vdash \varphi$ for all finitely generated subframes $\fG$ of $\fF$, 
then $\fF \Vdash \varphi$. 
\end{proposition}
\begin{proof}
We show the contrapositive implication. Let $\varphi$ be $\bigvee_{i \in I} \uBox \psi_{i}$ and suppose that $\fF \not\Vdash {\bigvee_{i \in I}} \uBox \psi_{i}$. Now, we can find a valuation $V$ and a state $w$ such that $(\fF,V),w \not\Vdash \uBox \psi_{i}$ for all $i \in I$. Thus, for each $i \in I$, there is a state $w_{i}$ such that $(\fF,V),w_{i} \not\Vdash \psi_{i}$. Define $X$ := $\inset{w_{i}}{i \in I}$ and note that $X$ is finite.
Consider the submodel $(\fF_{X},V_{X})$ of $\fF$ generated by $X$. 
Since for each $i\in I$, $(\fF,V),w_{i} \not\Vdash \psi_{i}$ and $\psi_{i}\in\ML$, and since the satisfaction of $\ML$ is invariant under generated submodels (cf.~\cite[Proposition 2.6]{blackburn:2001}), it follows that $(\fF_{X},V_{X}),w_{i} \not\Vdash \psi_{i}$ for each $i\in I$. Thus $(\fF_{X},V_{X}) \not\Vdash \uBox \psi_{i}$ for each $i \in I$.
Hence $(\fF_{X},V_{X}) \not\Vdash \bigvee_{i \in I} \uBox \psi_{i}$, which implies our goal $\fF_{X} \not\Vdash \bigvee_{i \in I} \uBox \psi_{i}$.
\end{proof}

The fact that every $\ML(\uBoxp)$-definable class reflects finitely generated subframes follows by Propositions \ref{prop:mlup_eq_vuml} and \ref{prop:fin_gen_frames}. 

\begin{proposition}\label{uboxreflects}
Every $\ML(\uBoxp)$-definable class of Kripke frames reflects finitely generated subframes.
\end{proposition}


%
%
Whereas the original Goldblatt--Thomason theorem for basic modal logic was proved via duality between algebras and frames~\cite{GT1975}, our proof of Goldblatt--Thomason -style theorem modifies the model-theoretic proof given by van Benthem~\cite{Benthem1993} for basic modal logic.

\begin{definition}[Satisfiability]
Let $\Gamma$ be a set of formulas, $\mK$ a model and $\cF$ a class of frames. 
We say that $\Gamma$ is {\em satisfiable} in $\mK$ if there exists a point $w$ of $\mK$ such that $\mK,w\Vdash \gamma$ for all $\gamma \in \Gamma$. We say that $\Gamma$ is {\em finitely satisfiable} in $\mK$ if each finite subset of $\Gamma$ is satisfiable in $\mK$. We say that $\Gamma$ is {\em satisfiable} in $\cF$ if there exists a frame $\fF\in \cF$ and a valuation $V$ on $\fF$ such that $\Gamma$ is satisfiable in $(\fF, V)$. Finally, we say that $\Gamma$ is {\em finitely satisfiable} in $\cF$ if each finite subset of $\Gamma$ is satisfiable in $\cF$.
\end{definition}

\begin{theorem}
\label{GbTh4mlup}
Given any elementary frame class $\cF$, the following are equivalent:
\begin{itemize}
\item[$(\mathrm{i})$] $\cF$ is $\ML(\uBoxp)$-definable.
\item[$(\mathrm{ii})$] $\cF$ is closed under taking generated subframes and bounded morphic images, and reflects ultrafilter extensions and finitely generated subframes. 
\end{itemize}
\end{theorem}

\begin{proof}
The direction from $(\mathrm{i})$ to $(\mathrm{ii})$ follows directly by Propositions \ref{uboxgen} and \ref{uboxreflects}, and Theorem \ref{gbth}.
In the proof of the converse direction, we use some notions from first-order model theory such as elementary extensions and $\omega$-saturation. The reader unfamiliar with them is referred to~\cite{CK1990}. Assume $(\mathrm{ii})$ and define $\mathrm{Log}(\cF)$ := $\inset{\varphi \in \ML(\uBoxp)}{\cF \Vdash \varphi}$. 
We show that, for any frame $\fF$, 
$\fF \in \cF$  iff  $\fF \Vdash \mathrm{Log}(\cF)$.

Consider any $\fF=(W,R)$. It is trivial to show the Only-If-direction, and so we show the If-direction. 
Assume that $\fF \Vdash \mathrm{Log}(\cF)$. 
To show $\fF \in \cF$, we may assume, without loss of generality, that $\fF$ is finitely generated. 
This is because: otherwise, it would suffice to show, since $\cF$ reflects finitely generated subframes, that $\fG \in \cF$ for all finitely generated subframes $\fG$ of $\fF$ (note that $\fG \Vdash \mathrm{Log}(\cF)$ by Proposition \ref{uboxgen}). 
Let $U$ be a finite generator of $\fF$. 
Let us expand our syntax with a (possibly uncountable) set $\inset{p_{A}}{A \subseteq W}$ of new propositional variables and define $\Delta$ to be the set containing exactly:
\[
p_{A \cap B}\leftrightarrow p_{A} \land p_{B}, \quad
p_{W \setminus A}\leftrightarrow\neg p_{A}, \quad
p_{m_{R}(A)}\leftrightarrow  \Diamond p_{A}, \quad
p_{W}, \\
\]
where $A, B \subseteq W$ and recall that $m_{R}(A)$ := $\inset{x \in W}{\text{$xRy$ for some $y \in A$ }}$ (cf.~Definition~\ref{dfn:ultra_filter_ex}). 
Define
\[
\Delta_{\fF,u} \dfn \{p_{\{u\}} \land \Box^n \varphi \mid n\in\omega \text{ and } \varphi \in \Delta\}, 
\]
for each $u\in U$. 
Recall that $\fF$ is finitely generated by $U$. The intuition here is that $(\Delta_{\fF,u})_{u \in U}$ provides a ``complete enough description'' of $\mathfrak{F}$. 

Let us introduce a finite set $\{x_{u} | u \in U \}$ of variables in first-order syntax and let $ST_{x_{u}}$ be the standard translation from $\ML(\uBoxp)$ to the corresponding first-order logic via the variable $x_{u}$, see Definition \ref{def:st}. We will show that $\bigcup_{u \in U} \{ ST_{x_{u}} (\varphi) \,|\, \varphi \in \Delta_{\fF,u}\}$ is satisfiable in $\cF$ in the sense of the satisfaction in first-order model theory. Since $\cF$ is elementary, it follows from the compactness of first-order logic that it suffices to show that $\bigcup_{u \in U} \{ ST_{x_{u}} (\varphi) \,|\, \varphi \in \Delta_{\fF,u}\}$ is finitely satisfiable in $\cF$. Let $\Gamma$ be a finite subset of this set. Then, we may write $\Gamma$ = $\bigcup_{1 \leq k \leq n} ST_{x_{u_{k}}} [\Gamma_{u_{k}}]$ for some $u_{1}$, \ldots, $u_{n} \in U$ and some finite $\Gamma_{u_{k}} \subseteq \Delta_{\fF,u_{k}}$ ($1 \leq k \leq n$). 
Assume, for the sake of a contradiction, that $\Gamma$ is not satisfiable in $\cF$. 
It follows that $\cF \Vdash \vartheta$ in the sense of modal logic, where $\vartheta$ := $\bigvee_{1 \leq k \leq n} \uBox \neg \bigwedge \Gamma_{u_{k}}$. 
Since $\vartheta$ is an $\ML(\uBoxp)$-formula, it belongs to $\mathrm{Log}(\cF)$. 
Thus by the assumption $\fF  \Vdash \mathrm{Log}(\cF)$, we conclude that $\fF \Vdash\vartheta$; 
and therefore $\Gamma$ is not satisfiable in $\mathfrak{F}$ in the sense of first-order model theory. %
However, $\Gamma$ is clearly satisfiable in $\fF$ under the natural structure interpreting $p_{A}$ as $A$ and the natural assignment sending $x_{u}$ to $u$. 
This is a contradiction.
Therefore, $\bigcup_{u \in U} \{ ST_{x_{u}} (\varphi) \,|\, \varphi \in \Delta_{\fF,u}\}$ is satisfiable in $\cF$.

Let $\fG \in \cF$ be such that $\bigcup_{u \in U} \{ ST_{x_{u}} (\varphi) \,|\, \varphi \in \Delta_{\fF,u}\}$ is satisfiable in $\fG$. 
Let us fix a valuation $V$ and a finite set $Z$ := $\{w_{u}  | u \in U\}$ of points such that 
$\bigcup_{u \in U} \{ ST_{x_{u}} (\varphi) \,|\, \varphi \in \Delta_{\fF,u}\}$ is satisfied in $(\fG,V)$ under an assignment sending each $x_{u}$ to $w_{u}$.
Then, $(\fG,V),  w_{u} \Vdash  \Delta_{\fF,u}$. 
Now let $(\fG^*_Z,V^*_Z)$ denote some $\omega$-saturated elementary extension of the $Z$ generated submodel of $(\fG,V)$. 
It is easy to check that $(\fG^*_Z,V^*_Z),  w_{u}^{\ast} \Vdash  \Delta_{\fF,u}$ where $w_{u}^{\ast}$ is the corresponding element in $\fG^*_Z$ to $w_{u}$ of $\fG_Z$ and that $(\fG^*_Z,V^*_Z)\Vdash \Delta$. 
Since $\cF$ is elementary and closed under taking generated subframes, we conclude first that $\fG_{Z} \in \cF$ and then that $\fG^*_Z\in \cF$.
%
%
%
%
%
%
%
We can now prove the following claim. 
\begin{claim}
The ultrafilter extension 
$\mathfrak{ue F}$ is a bounded morphic image of $\fG_{Z}^{\ast}$. 
\end{claim}
\noindent By closure of $\cF$ under bounded morphic images, we oftain $\mathfrak{ueF} \in \cF$. Finally, since $\cF$ reflects ultrafilter extensions, $\fF \in \cF$, as required. \qed

\noindent \textbf{(Proof of $Claim$)} 
Define a mapping $f: |\fG_{Z}^{\ast}| \to \mathrm{Uf}(W)$ (where $\mathrm{Uf}(W)$ is the set of all ultrafilters on $W$) by 
\[
f(s) := \inset{A \subseteq W}{ (\fG_{Z}^{\ast},V_{Z}^{\ast}), s \Vdash p_{A}  }.
\]
We will show that (a) $f(s)$ is an ultrafilter on $W$; (b)$f$ is a bounded morphism; (c) $f$ is surjective. 
Below, we denote by $S$ the underlying binary relation of $\fG_{Z}^{\ast}$. 
\begin{itemize}
\item[(a)] $f(u)$ is an ultrafilter: Follows immediately from the fact that $(\fG_{Z}^{\ast},V_{Z}^{\ast}) \Vdash \Delta$. 
\item[(b1)] $f$ satisfies $\textbf{(Forth)}$: We show that $sSs'$ implies $f(s)R^{\mathfrak{ue}}f(s')$. 
Assume that $sSs'$. 
By the definition of $R^{\mathfrak{ue}}$, it suffices to show that $A \in f(s')$ implies $m_{R}(A) \in f(s)$. 
Suppose $A \in f(s')$. Thus $(\fG_{Z}^{\ast},V_{Z}^{\ast}), s' \Vdash p_{A}$. Since $sSs'$, we obtain $(\fG_{Z}^{\ast},V_{Z}^{\ast}), s \Vdash \Diamond p_{A}$. Since $(\fG_{Z}^{\ast},V_{Z}^{\ast}) \Vdash \Delta$, $(\fG_{Z}^{\ast},V_{Z}^{\ast}) \Vdash \Diamond p_{A} \leftrightarrow p_{m_{R}(A)}$. 
Therefore $(\fG_{Z}^{\ast},V_{Z}^{\ast}), s \Vdash p_{m_{R}(A)}$, and hence $m_{R}(A) \in f(s)$, as desired. 

\item[(b2)] $f$ satisfies $\textbf{(Back)}$: We show that $f(s)R^{\mathfrak{ue}}\mathcal{U}$ implies $sSs'$ and $f(s')$ = $\mathcal{U}$ for some $s' \in |\fG_{Z}^{\ast}|$. Assume that $f(s)R^{\mathfrak{ue}}\mathcal{U}$. 
We will find a state $s'$ such that $sSs'$ and $(\fG_{Z}^{\ast},V_{Z}^{\ast}), s' \Vdash p_{A}$ for all $A \in \mathcal{U}$. 
By $\omega$-saturation, it suffices to show that $\inset{p_{A}}{A \in \mathcal{U}}$ is finitely satisfiable in the set $\inset{t \in |\fG_{Z}^{\ast}|}{sSt}$ of the successors of $s$. 
Take any $A_{1}$, $\ldots$, $A_{n} \in \mathcal{U}$. 
Then, $\bigcap_{1 \leq i \leq n} A_{i} \in \mathcal{U}$. 
Now since $f(s)R^{\mathfrak{ue}}\mathcal{U}$, $m_{R}(\bigcap_{1 \leq i \leq n} A_{i}) \in f(s)$. 
Hence $(\fG_{Z}^{\ast},V_{Z}^{\ast}), s \Vdash p_{m_{R}(\bigcap_{1 \leq i \leq n} A_{i})}$.
Since $(\fG_{Z}^{\ast},V_{Z}^{\ast}) \Vdash \Delta$, $(\fG_{Z}^{\ast},V_{Z}^{\ast}) \Vdash p_{m_{R}(\bigcap_{1 \leq i \leq n} A_{i})} \leftrightarrow \Diamond p_{\bigcap_{1 \leq i \leq n} A_{i}}$.
Therefore $(\fG_{Z}^{\ast},V_{Z}^{\ast}), s \Vdash \Diamond p_{\bigcap_{1 \leq i \leq n} A_{i}}$. 
Thus there is a state $s' \in |\fG_{Z}^{\ast}|$ such that $sSs'$ and 
$(\fG_{Z}^{\ast},V_{Z}^{\ast}), s' \Vdash p_{\bigcap_{1 \leq i \leq n} A_{i}}$. Therefore and since $(\fG_{Z}^{\ast},V_{Z}^{\ast}) \Vdash \Delta$, it follows that 
$(\fG_{Z}^{\ast},V_{Z}^{\ast}), s' \Vdash p_{A_{i}}$ for all $1 \leq i \leq n$.
\item[(c)] $f$ is surjective: Let us take any ultrafilter $\mathcal{U} \in |\mathfrak{ueF}|$. 
To prove surjectiveness, we show that the set $\inset{p_{A}}{A \in \mathcal{U}}$ is satisfiable in $(\fG_{Z}^{\ast},V_{Z}^{\ast})$. By $\omega$-saturatedness of $(\fG_{Z}^{\ast},V_{Z}^{\ast})$, it suffices to show finite satisfiability. Fix any $A_{1}, \ldots, A_{n} \in \mathcal{U}$. 
It follows that $\bigcap_{1 \leq k \leq n} A_{k} \in \mathcal{U}$, and hence $\bigcap_{1 \leq k \leq n} A_{k} \neq \emptyset$. 
Pick $w \in \bigcap_{1 \leq k \leq n} A_{k}$. 
Since $\fF$ is finitely generated by $U$, $w$ is reachable (in $\fF$) from some point $u \in U$ in a finite number of steps. But then there is some $l \in \omega$ such that $(\fF, V_{0}), u \Vdash p_{ (m_{R})^{l} ( \bigcap_{1 \leq k \leq n} A_{k}  )}$, where $V_{0}$ is the natural valuation on $\fF$ sending $p_{X}$ to $X$. 
%
Since $V_0$ is the natural valuation, we also obtain that $u\in  (m_{R})^{l}( \bigcap_{1 \leq k \leq n} A_{k})$,
and thus $\Delta$ contains $p_{\setof{u}} \leftrightarrow p_{\setof{u}} \land p_{(m_{R})^{l} ( \bigcap_{1 \leq k \leq n} A_{k})}$. It now follows from $(\fG^*_Z,V^*_Z),  w_{u}^{\ast} \Vdash  \Delta_{\fF,u}$ that 
$(\fG_{Z}^{\ast},V_{Z}^{\ast}), w_{u}^{\ast} \Vdash p_{\setof{u}}$.
Since $(\fG_{Z}^{\ast},V_{Z}^{\ast}) \Vdash \Delta$, we obtain
$(\fG_{Z}^{\ast},V_{Z}^{\ast}), w_{u}^{\ast} \Vdash p_{ (m_{R})^{l} ( \bigcap_{1 \leq k \leq n} A_{k})}  $, and hence also that
$(\fG_{Z}^{\ast},V_{Z}^{\ast}), w_{u}^{\ast} \Vdash \Diamond^{l} p_{\bigcap_{1 \leq k \leq n} A_{k}}$. 
Therefore, $\setof{p_{A_{1}},\ldots,p_{A_{n}}}$ is satisfiable in $(\fG_{Z}^{\ast},V_{Z}^{\ast})$. \hfill $\dashv$
\end{itemize}
\let\qed\relax
\end{proof}

\section{Finite Goldblatt-Thomason-style Theorem for Relative Modal definability with Positive Universal Modality}\label{sec:relmd}

Given a class $\cG$ of frames,  we say that a set of formulas {\em defines} a class $\cF$ of frames {\em within} $\cG$ if, for all frames $\fF \in \cG$, the equivalence: $\fF \Vdash \varphi$ $\Leftrightarrow$ $\fF \in \cF$ holds. A frame $\cF=(W,R)$ is called \emph{finite} whenever $W$ is a finite set and \emph{transitive} whenever $R$ is a transitive relation.  
In what follows, let $\cF_{\mathrm{fintra}}$ be  the class of all finite transitive frames and $\cF_{\mathrm{fin}}$ the class of all finite frames.

With the help of frame constructions such as bounded morphic images, disjoint unions, generated subframes, we first review the existing characterisations of relative $\ML$- and $\ML(\uBox)$-definability within the class of finite transitive frames. We then give a novel characterisation of relative $\ML(\uBoxp)$-definability again within the class of finite transitive frames. 

\begin{theorem}[Finite Goldblatt--Thomason Theorems for $\ML$ \cite{Benthem1993} and $\ML(\uBox)$ \cite{GG1993}]
\label{fingbth}
\text{}

\begin{enumerate}
\item A class of finite transitive frames is $\ML$-definable within the class $\cF_{\mathrm{fintra}}$ of all finite transitive frames if and only if it is closed under taking bounded morphic images, generated subframes, and disjoint unions. 
\item A class of finite frames is $\ML(\uBox)$-definable within the class $\cF_{\mathrm{fin}}$ of all finite frames if and only if it is closed under taking bounded morphic images. 
\end{enumerate}
\end{theorem}

In order to show the corresponding characterisation of relative definability in $\ML(\uBoxp)$, a variant of the Jankov-Fine formula is defined. 

\begin{definition}
Let $\fF$ = $(W,R)$ be a finite transitive frame. Put $W$ $:=$ $\setof{w_{0},\ldots,w_{n}}$. Associate a new proposition variable $p_{w_{i}}$ with each $w_{i}$ and define $\Box^{+} \varphi$ := $\Box \varphi \land \varphi$. The Jankov-Fine formula $\varphi_{\fF,w_{i}}$ at $w_{i}$ is defined as the conjunction of all the following formulas: 
\begin{enumerate}
\item $p_{w_{i}}$
\item $\Box(p_{w_{0}} \lor \cdots \lor p_{w_{n}})$. 
\item $\bigwedge \inset{\Box^{+}( p_{w_i} \to \neg p_{w_j} )}{w_{i} \neq w_{j}}$.
\item $\bigwedge \inset{\Box^{+}( p_{w_i} \to \Diamond  p_{w_j} )}{ (w_{i},w_{j}) \in R }$.
\item $\bigwedge \inset{\Box^{+}( p_{w_i} \to \neg \Diamond  p_{w_j} )}{(w_{i},w_{j}) \notin R }$.
\end{enumerate}
The Jankov-Fine formula $\varphi_{\fF}$  is defined as $\bigvee_{w \in W} \uBox \neg \varphi_{\mathfrak{F},w}$.
\end{definition}

We note that the Jankov-Fine formula $\varphi_{\fF,w_{i}}$ at $w_{i}$ is an $\ML$-formula and thus the Jankov-Fine formula $\varphi_{\fF}$ is an $\ML(\uBoxp)$-formula. 

\begin{lemma}
\label{lem:fin_des}
Let $\mathfrak{F}$ = $(W,R)$ be a finite transitive frame. For any transitive frame $\fG$, the following are equivalent:
\begin{itemize}
\item[$(\mathrm{i})$] the Jankov-Fine formula $\varphi_{\fF}$ is not valid in $\fG$, 
\item[$(\mathrm{ii})$]  there is a finite set $Y \subseteq |\fG|$ such that $\fF$ is a bounded morphic image of $\fG_{Y}$, where $\fG_{Y}$ is the subframe of $\fG$ generated by $Y$. 
\end{itemize}
\end{lemma}
\begin{proof}
The direction from (ii) to (i) is immediate from the fact that $\varphi_{\mathfrak{F}}$ is not valid in $\fF$ under the natural valuation sending $p_{w_{i}}$ to $\setof{w_{i}}$ (Note: validity of $\ML(\uBoxp)$-formulas is closed under taking under bounded morphic images and generated subframes, see, e.g., \cite[Exercise 7.1.2]{blackburn:2001} and Proposition \ref{uboxgen}).  
So, we focus on the converse direction.

Assume (i). It follows from $\fG\not\Vdash\varphi_{\fF}$ that  $(\fG, V)\not\Vdash\varphi_{\fF}$, for some assignment $V$. Thus,  for each $i\leq n$, there exists a point $v_i$ of $\fG$ such that $(\fG,V),v_i\Vdash\varphi_{\fF,w_i}$. Put $Y$ := $\inset{v_{i}}{0 \leq i\leq n}$, let $\fG_{Y}$ denote the subframe of $\fG$ generated by $Y$, and let $U$ be the reduction of $V$ into the frame $\fG_{Y}$. Since satisfaction of $\ML$-formulas is closed under taking generated submodels (see, e.g., \cite[Prop. 2.6]{blackburn:2001}), it follows that  $(\fG_{Y},U), v_{i} \Vdash \varphi_{\fF,w_{i}}$, for each $i\leq n$.
Let us put $\fG_{Y}$ = $(G_Y,S)$.
The first clause of the Jankov-Fine formula $\varphi_{\fF,w_{i}}$ implies that, for  each $i\leq n$, $U(p_{w_{i}})$ $\neq$ $\emptyset$.
By the second and the third clause, we obtain $\bigcup_{w \in W } U(p_{w})$ = $G_Y$ and $U(p_{w_{i}}) \cap U(p_{w_{j}}) = \emptyset$ for any distinct indices $i$ and $j$.  
This enables us to define a surjective mapping $f: G_Y \to W$. Define $f(v)\dfn w_i$ if $v\in U(p_{w_i})$. Clearly $f$ is a well defined surjection.

In what follows, we show that $f$ is a bounded morphism. The condition (Forth) is established as follows. Assume that $xSy$ and let $i,j$ be such that  $f(x)$ = $w_{i}$ and $f(y)$ = $w_{j}$. Thus $x \in U(p_{w_i})$ and $y \in U(p_{w_j})$. Since  $\fG_{Y}$ is $Y$-generated, $x$ is reachable from some $v_{k} \in Y$. Suppose for a contradiction that $w_{i}Rw_{j}$ fails in $\fF$. Then $\Box^{+}(p_{w_{i}} \to \neg \Diamond p_{w_{j}})$ is a conjunct in the Jankov-Fine formula $\varphi_{\fF,w_{k}}$. Recall that $(\fG_{Y},U), v_{k} \Vdash \varphi_{\fF,w_{k}}$. It now follows from $(\fG_{Y},U),v_{k} \Vdash \Box^{+}(p_{w_{i}} \to \neg \Diamond p_{w_{j}})$ that $xSy$ fails. A contradiction. Therefore, $w_{i}Rw_{j}$ holds in $\fF$. 

\noindent The condition (Back) is shown as follows. Assume that $f(x)R w_{j}$ and let $i$ be such that $f(x)$ = $w_{i}$. From the definition of $f$, it follows that  $x \in U(p_{w_i})$. Since  $\fG_{Y}$ is $Y$-generated, $x$ is reachable from some $v_{k} \in Y$. Since $w_{i}Rw_{j}$, we have that $\Box^{+}(p_{w_{i}} \to  \Diamond p_{w_{j}})$ is a conjunct in the Jankov-Fine formula $\varphi_{\fF,w_{k}}$.  Recall again that $(\fG_{Y},U), v_{k} \Vdash \varphi_{\fF,w_{k}}$. It follows from $(\fG_{Y},U),v_{k} \Vdash \Box^{+}(p_{w_{i}} \to \Diamond p_{w_{j}})$ and  $x \in U(p_{w_i})$ that there is some $y$ such that $f(y)$ = $w_{j}$ and $xSy$ holds, as desired.
\end{proof}

\begin{theorem}
\label{thm:GbTh4mlup}
For every class $\cF$ of finite transitive frames, the following are equivalent:
\begin{itemize}
\item[$(\mathrm{i})$] $\cF$ is $\ML(\uBoxp)$-definable within $\cF_{\mathrm{fintra}}$.
\item[$(\mathrm{ii})$] $\cF$ is closed under taking generated subframes and bounded morphic images.
\end{itemize}
\end{theorem}

\begin{proof}
The direction from (i) to (ii) is easy to establish (by Proposition \ref{uboxgen} and Theorem \ref{fingbth}), so we focus on the converse direction. Assume (ii). Define $\mathrm{Log}(\cF)$ = $\inset{\varphi \in \ML(\uBoxp)}{\cF \Vdash \varphi}$. We show that $\mathrm{Log}(\cF)$ defines $\cF$ within $\cF_{\mathrm{fintra}}$. Fix any finite and transitive frame $\fF \in \cF_{\mathrm{fintra}}$. In what follows, we show the following equivalence:
\[
\begin{array}{lll}
\fF \in \cF &\iff& \fF \Vdash \mathrm{Log}(\cF). \\
\end{array}
\]
The left-to-right direction is immediate, so we concentrate on the converse direction. Assume $\fF \Vdash \mathrm{Log}(\cF)$. Since $\mathfrak{F}$ is finite and transitive, let us take the Jankov-Fine formula $\varphi_{\mathfrak{F}}$. Since $\varphi_{\mathfrak{F}}$ is not valid in $\mathfrak{F}$, $\varphi_{\mathfrak{F}} \notin \mathrm{Log}(\cF)$. Thus there is a {\em transitive} frame $\fG \in \cF$ (recall that $\cF$ is a class of transitive frames) such that $\varphi_{\mathfrak{F}}$ is not valid in $\fG$.  By Lemma \ref{lem:fin_des}, there is a finite set $Y \subseteq |\fG|$ such that $\fF$ is a bounded morphic image of $\fG_{Y}$. Since $\fG \in \cF$, $\fG_{Y} \in \cF$ by $\cF$'s closure under generated subframes. It follows from $\cF$'s closure under bounded morphic images that $\fF \in \cF$, as desired.
\end{proof}

\section{Modal Logics with Team Semantics}\label{sec:mlteams}
We now turn from modal logics with Kripke semantics to modal logics in which the semantics is defined with respect to team-pointed models.
In this section we define the team-based modal logics that are relevant for this paper. We survey basic properties and known result concerning  expressive power. Later, in Section \ref{sec:mlandum}, we connect these two different semantics with respect to definability.

\subsection{Basic notions of team semantics}

A subset $T$ of the domain of a Kripke model $\mK$ is called \emph{a team of $\mK$}. Before we define the so-called \emph{team semantics} for $\ML$, let us first introduce some notation that makes defining the semantics simpler.
\begin{definition}
Let $\mK=(W,R,V)$ be a model and $T$ and $S$ teams of $\mK$. Define 
\begin{center}
$R[T]$ $:=$ $\{w\in W \mid \exists v\in T (vRw) \}$ and 
$R^{-1}[T]$ $:=$ $\{w\in W \mid \exists v\in T (wRv)\}$.
\end{center}
For teams $T$ and $S$ of $\mK$, we write $T[R]S$ if $S\subseteq R[T]$ and $T\subseteq R^{-1}[S]$.
\end{definition}
Thus, $T[R]S$ holds if and only if for every $w\in T$ there exists some $v\in S$ such that $wRv$, and for every $v\in S$ there exists some $w\in T$ such that $wRv$.
The team semantics for $\ML$ is defined as follows.
We use the symbol ``$\models$'' for team semantics instead of the symbol ``$\Vdash$'' which was used for Kripke semantics. 
\begin{definition}
Let $\mK$ be a Kripke model and $T$ a team of $\mK$. 
The satisfaction relation $\mK,T\models \varphi$ for $\ML(\Phi)$ is defined as follows. 
\begin{align*}
\mK,T\models p  \quad\Leftrightarrow\quad& w\in V(p) \, \text{ for every $w\in T$.}\\
\mK,T\models \neg p \quad\Leftrightarrow\quad& w\not\in V(p) \, \text{ for every $w\in T$.}\\
\mK,T\models (\varphi\land\psi) \quad\Leftrightarrow\quad& \mK,T\models\varphi \text{ and } \mK,T\models\psi.\\
\mK,T\models (\varphi\lor\psi) \quad\Leftrightarrow\quad& \mK,T_1\models\varphi \text{ and } 
\mK,T_2\models\psi \, \text{ for some $T_1$ and $T_2$}\\
&\text{such that $T_1\cup T_2= T$}.\\
\mK,T\models \Diamond\varphi \quad\Leftrightarrow\quad& \mK,T'\models\varphi \text{ for some $T'$ such that $T[R]T'$}.\\
\mK,T\models \Box\varphi \quad\Leftrightarrow\quad& \mK,T'\models\varphi, \text{ where $T'=R[T]$}.
\end{align*}
\end{definition}
A set $\Gamma$ of formulas is \emph{valid in a model} $\mK=(W,R,V)$ (in team semantics), in symbols $\mK\models \Gamma$, if $\mK,T\models \varphi$ holds for every team $T$ of $\mK$ and every $\varphi \in \Gamma$.  Likewise, we say that  $\Gamma$ is \emph{valid in a Kripke frame} $\fF$ and write $\fF\models\Gamma$, if  $(\fF,V)\models\Gamma$ hold for every valuation $V$. When $\Gamma$ is a singleton $\setof{\varphi}$, we simply write $\mK\models \varphi$ and $\fF\models\varphi$.
%

%

%

The formulas of $\ML$ have the following flatness property, see, e.g., \cite{dukovo16}.
\begin{proposition}[Flatness]\label{flatness}
\label{mlextends}
Let $\mK$ be a Kripke model and $T$ be a team of $\mK$. Then, for every formula $\varphi$ of $\ML(\Phi)$:
\(
\mK,T\models \varphi \,\text{ iff }\, \forall w\in T:\mK,w\Vdash \varphi.
\)
\end{proposition}
From flatness if follows that for every model $\mK$, frame $\fF$, and formula $\varphi$ of $\ML$, $\mK \Vdash \varphi$ iff $\mK\models \varphi$ and $\fF \Vdash \varphi$ iff $\fF\models \varphi$.

Recall from \Cref{sec:definability} what it means that a set of modal formulas defines a class of frames and models. All the related definitions can be adapted for logics with team semantics by simply substituting $\Vdash$ by $\models$.

The most important closure properties in the study of team-based logics are downward closure, union closure, and the concept of team bisimulation.
\begin{definition}
Let $\LL$ be some team-based modal logic, $\mK$ a Kripke model, and $T,S$ teams of $\mK$. We say that a formula $\varphi\in\LL$ is
\begin{enumerate}
\item \emph{downward closed} {if} $\mK, T \models\varphi$, whenever $\mK,S\models \varphi$ and $T\subseteq S$.
\item  \emph{union closed} {if} $\mK, T\cup S \models\varphi$, whenever $\mK, T\models \varphi$ and $\mK, S\models \varphi$.
\end{enumerate}  
A logic $\LL$ is called  \emph{downward closed} $($\emph{union closed}$)$ if every formula $\varphi\in\LL$ is {downward closed} $(${union closed}$)$.
We say that $\LL$ has the \emph{empty team} property, if $\mK,\emptyset \models \varphi$ holds for every model $\mK$ and every formula $\varphi\in\LL$.
\end{definition}
Team bisimulation and its finite approximation team $k$-bisimulation can be defined via the corresponding concepts of ordinary modal logic. In the definition below, we denote by $\bisim$ and $\kbisim$ the notions of bisimulation and $k$-bisimulation of ordinary modal logic (see, e.g., Definition \ref{def:bisim} and \cite[Definition 2.30]{blackburn:2001}), respectively.

\begin{definition}
Let $\mK,T$ and  $\mK',T'$ be team-pointed Kripke models. We say that $\mK,T$ and  $\mK',T'$ are team bisimilar, and write $\mK,T\Bisim\mK',T'$ if
\begin{enumerate}
\item for every $w\in T$ there exist some $w'\in T'$ such that $\mK,w\bisim\mK',w'$, and
\item for every $w'\in T'$ there exist some $w\in T$ such that $\mK,w\bisim\mK',w'$.
\end{enumerate}
The \emph{team $k$-bisimulation} relation $\kBisim$ is defined analogously with $\bisim$ replaced by $\kbisim$.
\end{definition}

\subsection{Extensions of modal logic via connectives}
We first introduce two expressive extensions of modal logic: an extension by the so-called \emph{intuitionistic disjunction} and an extension by the so-called \emph{contradictory negation}. These two logics are of great interest, since with respect to expressive power the logics subsume all most studied team-based modal logics, in particular all of those defined in Section \ref{sec:atoms}.

\emph{Modal logic with intuitionistic disjunction} $\ML(\idis)(\Phi)$ is obtained by extending the syntax of $\ML(\Phi)$ by the grammar rule $\varphi \ddfn (\varphi \idis \varphi)$ with the following semantics:
\[
\mK,T\models (\varphi\varovee\psi) \quad\Leftrightarrow\quad \mK,T\models\varphi \text{ or } \mK,T\models\psi.
\]
\emph{Modal team logic} $\MTL(\Phi)$ is obtained by extending the syntax of $\ML(\Phi)$ by the contradictory negation, i.e., the grammar rule $\varphi \ddfn \cneg \varphi$ with the following semantics:
\[
\mK,T\models \cneg\varphi \quad\Leftrightarrow\quad \mK,T\not\models\varphi.
\]
The following theorem for $\ML(\varovee)$ was proven by Hella et al. \cite{HeLuSaVi14} and for $\MTL$ by Kontinen et al. \cite{komu15}
\begin{theorem}\label{teambisthm}
A class $\cC$ of team-pointed Kripke models is definable by a single formula of
\begin{enumerate}
\item $\ML(\varovee)$ iff $\cC$ is downward closed, closed under team $k$-bisimulation, for some $k\in\mathbb{N}$, and admits the empty team property;
\item $\MTL$ iff $\cC$ is closed under team $k$-bisimulation, for some $k\in\mathbb{N}$.  
\end{enumerate}
\end{theorem}

\subsection{Extensions of modal logic with atomic dependency notions}\label{sec:atoms}
The syntaxes of \emph{modal dependence logic} $\MDL(\Phi)$ and \emph{extended modal dependence logic} $\EMDL(\Phi)$ are obtained by extending the syntax of $\ML(\Phi)$ by the following grammar rule for each $n\in \omega$:
\begin{align*}
&\text{For $\MDL$:}\quad &&\varphi \ddfn \dep{\varphi_1,\dots,\varphi_n,\psi}, \text{ where $\varphi_1,\dots,\varphi_n,\psi\in \Phi$}.\\
&\text{For $\EMDL$:}\quad &&\varphi \ddfn \dep{\varphi_1,\dots,\varphi_n,\psi}, \text{ where $\varphi_1,\dots,\varphi_n,\psi\in\ML(\Phi)$}.
\end{align*}
The intuitive meaning of the (modal) dependence atom $\dep{\varphi_1,\dots, \varphi_n,\psi}$ is that the truth value of the formula $\psi$ is completely determined by the truth values of $\varphi_1,\dots, \varphi_n$. The formal definition is given below:
\begin{align*}
\mK,T\models \dep{\varphi_1,\dots,\varphi_n,\psi} \;\Leftrightarrow\;& \forall w,v\in T: \bigwedge_{1 \leq i \leq n}(\mK,\{w\}\models\varphi_i \Leftrightarrow \mK,\{v\}\models\varphi_i)\\
& \text{implies }(\mK,\{w\}\models\psi\Leftrightarrow \mK,\{v\}\models\psi).
\end{align*}
The syntax of \emph{modal inclusion logic} $\MINC(\Phi)$ and \emph{extended modal inclusion logic} $\EMINC(\Phi)$ is obtained by extending the syntax of $\ML(\Phi)$ by the following grammar rule for each $n\in \omega$:
\[
\varphi \ddfn {\varphi_1,\dots,\varphi_n \subseteq \psi_1,\dots,\psi_n}, \text{ where $\varphi_1,\psi_1,\dots,\varphi_n,\psi_n\in\ML(\Phi)$}.
\]
In the additional grammar rules above for $\MINC$, we require that the formulas $\varphi_1,\psi_1,\dots,\varphi_n,\psi_n$ are proposition symbols in $\Phi$.
The meaning of the (modal) inclusion atom $\varphi_1,\dots,\varphi_n \subseteq \psi_1,\dots,\psi_n$ is that the truth values that occur in a given team for the tuple $\varphi_1,\dots,\varphi_n$ occur also as truth values for the tuple $\psi_1,\dots\psi_n$. The formal definition is given below:
\begin{align*}
\mK,T\models& \varphi_1,\dots,\varphi_n \subseteq \psi_1,\dots,\psi_n \\
 &\Leftrightarrow \forall w\in T \exists v\in T: \bigwedge_{1 \leq i \leq n}(\mK,\{w\}\models\varphi_i \Leftrightarrow \mK,\{v\}\models\psi_i).
\end{align*}

With respect to expressive power the following are known, see, e.g., \cite{dukovo16, hest15}:
\[
\ML < \MDL <\EMDL = \ML(\varovee) < \MTL
\]
\[
\ML < \MINC < \EMINC < \MTL.
\]
Thus Theorem \ref{teambisthm} holds also for $\EMDL$. An analogous theorem for $\EMINC$ is by Hella and Stumpf \cite{hest15}.
\begin{theorem}[\cite{hest15}]\label{teambisthmeminc}
A class $\cC$ of team-pointed Kripke models is definable by a single formula of $\EMINC$ iff $\cC$ is union closed, closed under team $k$-bisimulation, for some $k\in\mathbb{N}$, and admits the empty team property.
\end{theorem} 

The fact that  $\MINC < \EMINC$ holds is known but no published proof is known by the authors. Thus we present one here.
\begin{proposition}
With respect to expressive power $\MINC < \EMINC$.
\end{proposition}
\begin{proof}
For $\varphi\in \MINC(\{p\})$, let $\varphi^*$ denote the $\ML(\{p\})$-formula obtained from $\varphi$ by substituting each inclusion atom in $\varphi$ by the formula $(p\lor\neg p)$. Since $p\subseteq p$ is essentially the only inclusion atom in $\MINC(\{p\})$, it is easy to see that, for every $\varphi\in \MINC(\{p\})$,  $\varphi$ and $\varphi^*$ are equivalent.

Let $\mK=(W,R,V)$ be a Kripke $\{p\}$-model such that $W=\{1,2,3\}$, R=\{(1,2)\}, and $V(p)=\{1,2,3\}$. We claim that there does not exists a $\MINC$-formula that is equivalent with $p\subseteq \Diamond p$. For the sake of a contradiction, assume that $\psi\in \MINC$ is such a formula. Clearly $\mK,\{1,3\}\models p\subseteq \Diamond p$ and thus, by assumption, $\mK,\{1,3\}\models \psi$. By our observation above, $\mK,\{1,3\}\models \psi^*$ follows. Now since $\psi^*$ is an \ML-formula, it follows by \Cref{flatness} that $\mK,\{3\}\models \psi^*$. Thus  $\mK,\{3\}\models \psi$ and therefore  $\mK,\{3\}\models p\subseteq \Diamond p$. However clearly $\mK,\{3\}\not\models p\subseteq \Diamond p$, a contradiction.
\end{proof}

The following proposition is proven in the same way as the analogous propositions for first-order dependence logic \cite{va07} and inclusion logic  \cite{Galliani:2011}.
\begin{proposition}[Closure properties]\label{closures}
The logics weaker or equal to $\ML(\varovee)$ with respect to expressive power are downward closed. The logics weaker or equal to $\EMINC$ with respect to expressive power are union closed.
\end{proposition}
Note that $\MTL$ is neither downward nor union closed. The modal depth of $\varphi$, denoted by $\md{\varphi}$, is defined in the obvious way (for basic modal logic, see e.g., \cite{blackburn:2001}); intuitionistic disjunction and contradictory negation are handled in the same  manner as Boolean connectives. For dependence atoms and inclusion atoms, we define that
\begin{align*}
\md{\dep{\varphi_1,\dots,\varphi_n,\psi}} &:= \max\{\md{\varphi_1}, \dots, \md{\varphi_n},\md{\psi}\},\\ 
\md{\varphi_1,\dots,\varphi_n \subseteq \psi_1,\dots,\psi_n} &:= \max\{\md{\varphi_1},\md{\psi_1}, \dots, \md{\varphi_n}, \md{\psi_n}\}.
\end{align*}
If $\LL$ is a logic and $k\in\mathbb{N}$, we write  $\mK,T \equiv^\LL_k   \mK',T'$, if $\mK,T$ and $\mK',T'$ agree on all $\LL$-formulas $\varphi$ with $\md{\varphi}\leq k$.
\begin{theorem}[\!\!\cite{komu15}]\label{bisthm}
Let $\LL$ be a team-based logic that is weaker or equal to  $\MTL$ with respect to expressive power. Then
\(
\mK,T\kBisim\mK',T'\, \Rightarrow \, \mK,T \equiv^\LL_k   \mK',T'.
\)
\end{theorem}
\section{Modal definability in team semantics}
The expressive power of the most studied team-based modal logics is quite well understood. See Table \ref{table:exp} for the known characterisations. However the closely related topics of definability with respect to models and with respect to frames has not been studied before. 

\begin{table}
\begin{center}
\begin{tabular}{cccccc}\toprule
Logic & & \multicolumn{3}{c}{Closure properties} & References \\ \cmidrule{3-5}
&empty team& team & downward & union & \\
&property&k-bisimulation & closure & closure & \\\midrule
\ML & X & X & X & X & \cite{Hennessy1985}\\
$\ML(\varovee)$ & X & X & X &  & \cite[C. 3.6]{HeLuSaVi14}\\
\EMDL & X & X & X &  & \cite[C. 4.5]{HeLuSaVi14}\\
\EMINC & X & X &  & X & \cite[T. 3.10]{hest15} \\
\MTL & & X &  &  & \cite[T. 3.4]{komu15}
\\\bottomrule
\end{tabular}
\caption{Characterisation of expressive powers of different team-based logics. E.g., a class $\cC$ of team-pointed Kripke models is definable by a single $\EMDL$-formula if and only if $\mK,\emptyset\in\ \cC$, for every $\mK$, $\cC$ is closed under the so-called team k-bisimulation, for some finite $k$, and $\cC$ is downward closed.}
\label{table:exp}
\end{center}
\end{table}

\begin{table}\centering
\begin{tabular}{c}\toprule
$\{ \ML, \MINC, \EMINC \} \mle \MDL \mle \{\EMDL, \ML(\varovee), \MTL \}$ \\
$\{ \ML, \MINC, \EMINC \} \fle \{\MDL, \EMDL, \ML(\varovee), \MTL \}$ \\
\bottomrule
\end{tabular}
\caption{Hierarchy of model and frame definability of different modal logics with team semantics. The logics within a same set above are proven to coincide with respect to model or frame definability.}
\label{table:definability}
\end{table}

Recall the hierarchy of the team-based logics with respect to expressive power stated in the previous section; there are six distinct cases. In this section we show that with respect to model definability there are only three distinct cases whereas with respect to frame definability only two different cases remain. See Table \ref{table:definability} for the resulting hierarchies. In order to show that, with respect to model definability, $\EMINC$ collapses to $\ML$ and that $\MTL$ collapses to $\ML(\varovee)$, we need to introduce the concepts of Hintikka types from model theory of modal logic.

\subsection{Hintikka formulas and types}
It is well-known that for any finite set of proposition symbols $\Phi$, any finite $k\in\mathbb{N}$, and any pointed $\Phi$-model $(\mK,w)$, there exists a modal formula of modal depth $k$ that characterises $(\mK,w)$ completely up to $k$-equivalence (i.e. equivalence up to modal depth $k$). These \emph{Hintikka formulas} (or \emph{characteristic formulas})
are defined as follows (see e.g. \cite{Goranko2007}):

\begin{definition}
Assume that $\Phi$ is a finite set of proposition symbols.
Let $k\in\mathbb{N}$ and let $(\mK,w)$ be a pointed $\Phi$-model. The \emph{$k$-th Hintikka formula}
$\chi^k_{\mK,w}$ of $(\mK,w)$ is defined recursively as follows:
\begin{itemize}
\item $\chi^0_{\mK,w}:=\bigwedge \{p\mid p\in \Phi, w\in V(p)\}\land\bigwedge\{\lnot p\mid p\in\Phi, w\not\in V(p)\}$.
\item $\chi^{k+1}_{\mK,w}:=\chi^k_{\mK,w}\land \bigwedge_{v\in R[w]}\Diamond\chi^k_{\mK,v}
\land\Box \bigvee_{v\in R[w]}\chi^k_{\mK,v}$.
\end{itemize}
\end{definition}

It is easy to see that $\md{\chi^k_{\mK,w}}=k$, and $\mK,w\Vdash\chi^k_{\mK,w}$
for every pointed $\Phi$-model $(\mK,w)$. By a straightforward inductive argument, it can be shown that $\chi^k_{\mK,w}$ is essentially finite. Moreover it can be shown that, for each fixed $k$ and $\Phi$, there exists only finitely many non-equivalent  $k$-th Hintikka formulas.

\begin{proposition}[see, e.g., \cite{Goranko2007}]\label{hintikka}
Let $\Phi$ be a finite set of proposition symbols, $k\in\mathbb{N}$, and $(\mK,w)$ and $(\mK',w')$ pointed $\Phi$-models.
Then
\[
	 \mK,w\equiv^\ML_k \mK',w'\quad\iff\quad \mK,w\kbisim \mK',w'\quad\iff\quad \mK',w'\Vdash\chi^k_{\mK,w}.
\]
\end{proposition}
Note that from above it follows that, up to equivalence, each pointed model $\mK,w$ satisfies exactly one $k$-th Hintikka formula, namely the formula $\chi^k_{\mK,w}$.
\begin{definition}
Let $\mK$ be a Kripke $\Phi$-model and $\cC$ a class of Kripke $\Phi$-models. We define that
\begin{align*}
\type_k(\mK) \dfn& \{ \htype^k_{\mK,w} \mid  \text{$w$ is a point of $\mK$}\},\\
\type_k(\mK,T) \dfn& \{ \htype^k_{\mK,w} \mid w\in T \},\\
\type_k(\cC) \dfn& \{ \type_k(\mK) \mid  \mK\in\cC \}.
\end{align*}
\end{definition}

\begin{proposition}\label{typetoeq}
Let $\LL$ be any team-based logic weaker than or equal to $\MTL$ w.r.t. expressive power. Then
\(
\type_k(\mK,T) = \type_k(\mK',T') \, \Rightarrow \, \mK,T \equiv^\LL_k   \mK',T'.
\)
\end{proposition}
\begin{proof}
Assume that $\type_k(\mK,T) = \type_k(\mK',T')$. By  Proposition $\ref{hintikka}$ and the definition of team bisimulation, it follows that  $\mK,T\kBisim\mK',T'$. The claim now follows by \Cref{bisthm}.
\end{proof}

\subsection{Definability with respect to models}
For sets of logics $\mathcal{A}$ and $\mathcal{B}$, we write $\mathcal{A}\mle \mathcal{B}$ if for each $\mathcal{L}_1,\mathcal{L}_2\in \mathcal{A}$ and $\mathcal{L}_3,\mathcal{L}_4\in \mathcal{B}$ it holds that $\mathcal{L}_1\meq \mathcal{L}_2\mle \mathcal{L}_3\meq \mathcal{L}_4$. For a singleton set $\{L\}$, we write simply $\mathcal{L}$.
The objective of this section is to prove the following trichotomy:
\[
\{ \ML, \MINC, \EMINC \} \mle \MDL \mle \{\EMDL, \ML(\varovee), \MTL \}.
\]
We will first establish that $\ML \mle \MDL \mle \EMDL$. We will then show that $\ML\meq\EMINC$ and finally that $\ML(\varovee)\meq \MTL$.
Since by the work of Hella et al. \cite{HeLuSaVi14} $\EMDL=\ML(\idis)$, already with respect to expressive power, the trichotomy follows. 

We will first show that $\ML \mle \MDL$ and that $\MDL \mle \EMDL$.
\begin{proposition}\label{prop:ml<mdl}
$\ML \mle \MDL$.
\end{proposition}
\begin{proof}
Let $\mK_i=(W_i,R_i,V_i)$, $i\leq 2$, be $\Phi$-models such that $W_0=\{1,2\}$,  $W_1=\{1\}$,  $W_2=\{2\}$, $R_0=R_1=R_2=\emptyset$, and, for each $p\in\Phi$, $V_0(p)=V_1(p)=\{1\}$, and $V_2(p)=\emptyset$. It is easy to conclude by flatness of $\ML$ that
\[
\mK_0\in\modlss(\varphi) \text{ iff } \mK_1,\mK_2\in\modlss(\varphi)
\]  
holds for every $\varphi\in\ML$. Thus 
\[
\mK_0\in\modlss(\Gamma) \text{ iff } \mK_1,\mK_2\in\modlss(\Gamma)
\]  
holds for every $\Gamma\subseteq\ML$. 
However $\mK_1,\mK_2\in \modlss(\dep{p})$ but  $\mK_0\not\in\modlss(\dep{p})$. Thus we conclude that $\modlss(\dep{p})$ is not definable in $\ML$.
\end{proof}

\begin{proposition}\label{prop:mdl<emdl}
$\MDL \mle \EMDL$.
\end{proposition}
\begin{proof}
Let $\mK_i=(W_i,R_i,V_i)$, $i\leq 2$, be $\Phi$-models such that $W_0=\{1,2\}$,  $W_1=\{1\}$,  $W_2=\{2\}$, $R_0=\{(1,1)\}$, $R_1=\{(1,1)\}$,  $R_2=\emptyset$, and, for each $p\in\Phi$, $V_0(p)=\{1,2\}$, $V_1(p)=\{1\}$, and $V_2(p)=\{2\}$.
It is easy to conclude (see \cite[Theorem 1]{EHMMVV13} for details) that
\[
\mK_0\in\modlss(\varphi) \text{ iff } \mK_1,\mK_2\in\modlss(\varphi)
\]  
holds for every $\varphi\in\MDL$. Thus 
\[
\mK_0\in\modlss(\Gamma) \text{ iff } \mK_1,\mK_2\in\modlss(\Gamma)
\]  
holds for every $\Gamma\subseteq\MDL$. 
However $\mK_1,\mK_2\in \modlss(\dep{\Diamond p})$ but  $\mK_0\not\in\modlss(\dep{\Diamond p})$. Thus we conclude that $\modlss(\dep{\Diamond p})$ is not definable in $\MDL$.
\end{proof}

We continue by establishing that every $\EMINC$-definable class of models is also definable in $\ML$.

\begin{lemma}\label{MINCbytypes}
Let $\Phi$ be a finite set of proposition symbols, $\varphi\in\EMINC(\Phi)$, and $k=\md{\varphi}$. Then
\(
 \mK\in\modlss(\varphi) \text{ iff }\type_k(\mK) \subseteq \bigcup\{\type_k(\mK') \mid \mK' \in \modlss(\varphi) \}.
\)
\end{lemma}
\begin{proof}
The direction from left to right is trivial. Assume then that
\begin{equation}\label{eq:1}
\type_k(\mK) \subseteq \bigcup\{\type_k(\mK') \mid \mK' \in \modlss(\varphi)\}
\end{equation}
holds, and let $T$ be an arbitrary team of $\mK$. It suffices to establish that $\mK,T\models\varphi$. From \eqref{eq:1} it follows that there exists some $n\in\mathbb{N}$, models $\mK_i\in\modlss(\varphi)$, teams $S_i$ of $\mK_i$ and $T_i$ of $\mK$, $i\leq n$, such that
\[
T_1\cup\dots\cup T_n = T \text{ and } \type_k(\mK_i,S_i)=\type_k(\mK,T_i),  \text{ for each $i\leq n$}.
\]
Note that such finite $n$ exists, since $\type_k(\mK)$ is essentially finite.
Since each $\mK_i\in\modlss(\varphi)$, it follows that $\mK_i,S_i\models \varphi$, for $i\leq n$. Thus from Proposition \ref{typetoeq} and the fact that $\type_k(\mK_i,S_i)=\type_k(\mK,T_i)$, for $i\leq n$, it follows that  $\mK,T_i\models \varphi$, for  $i\leq n$. Now, by union closure (\Cref{closures}), we conclude that  $\mK,T\models \varphi$.
\end{proof}

\begin{theorem}\label{MINCisMLs}
A class $\cC$ of Kripke models is definable by a single $\EMINC$-formula if and only if the class is definable by a single $\ML$-formula.
\end{theorem}
\begin{proof}
The if direction is trivial. For the other direction, let  $\cC$ be a class of Kripke models that is definable by a single $\EMINC$ formula and let $\varphi$ be an $\EMINC(\Phi)$-formula that defines $\cC$. Without lose of generality, we may assume that $\Phi$ is finite. Let $k$ denote the modal depth of $\varphi$. We will show that the $\ML(\Phi)$ formula
\[
\varphi^* \dfn \bigvee \{\chi^k_{\mK,w} \mid \mK\in \modlss(\varphi), w\in\mK  \}
\]
defines $\cC$.
Since over a finite set of proposition symbols there exists only finitely many essentially different $k$-Hintikka-formulas, $\varphi^*$ is essentially a finite $\ML(\Phi)$ -formula. By assumption $\cC=\modlss(\varphi)$. Thus by Lemma \ref{MINCbytypes}
\begin{equation}\label{eq:2}
\mK\in\cC  \text{ iff }\type_k(\mK) \subseteq \bigcup\{\type_k(\mK') \mid \mK' \in \modlss(\varphi) \}.
\end{equation}
Observe that by flatness (Proposition \ref{flatness}) and the fact that each pointed Kripke model satisfies only its own $k$-Hintikka-formula
\[
\mK,T \models \varphi^*  \text{ iff } \type_k(\mK,T) \subseteq \bigcup\{\type_k(\mK') \mid \mK' \in \modlss(\varphi) \},
\]
and thus it follows that
\begin{equation}\label{eq:3}
\mK \models \varphi^*  \text{ iff }\type_k(\mK) \subseteq \bigcup\{\type_k(\mK') \mid \mK' \in \modlss(\varphi) \}.
\end{equation}
From \eqref{eq:2} and \eqref{eq:3} the claim follows.
\end{proof}
The following theorem directly follow.
\begin{theorem}\label{MINCisML}
A class $\cC$ of Kripke models is $\EMINC$-definable  if and only if it is $\ML$ definable.
\end{theorem}

Finally we show that every $\MTL$ definable class of models is also definable in $\ML(\varovee)$.
\begin{lemma}\label{typelemma}
Let $\varphi$ be and $\MTL$-formula and $k=\md \varphi$. Then
\[
\mK \in \modlss(\varphi) \text{ iff } \type_k(\mK)\subseteq \Gamma\in \type_k\big(\modlss(\varphi)\big), \text{ for some $\Gamma$}.
\]
\end{lemma}
\begin{proof}
The direction from left to right is trivial. Assume then that  $\type_k(\mK)\subseteq \Gamma\in \type_k\big(\modlss(\varphi)\big)$ holds for some $\Gamma$. Thus there exists a Kripke model $\mK'$ such that $\mK' \in \modlss(\varphi)$ and  $\type_k(\mK')=\Gamma$. For the sake of a contradiction, assume that $\mK\not\in \modlss(\varphi)$. Thus there exists a team $T$ of $\mK$ such that $\mK,T\not\models\varphi$. Since $\type_k(\mK) \subseteq \type_k(\mK')$ it follows that there exists a team $T'$ of $\mK'$ such that  $\type_k(\mK,T) = \type_k(\mK',T')$. Thus by Proposition \ref{typetoeq}, we conclude that  $\mK',T'\not\models\varphi$. This is a contradiction and thus $\mK\in \modlss(\varphi)$ holds.
\end{proof}

\begin{theorem}\label{th:modeldefinable}
A class $\cC$ of Kripke models is definable in $\MTL$ by a single formula if and only if it is definable in $\ML(\varovee)$ by a single formula.
\end{theorem}
\begin{proof}
The fact that every class of Kripke models that is definable by a single $\ML(\varovee)$-formula is also definable by a single $\MTL$-formula follows directly by \Cref{teambisthm}.

Let $\cC$ be an arbitrary single formula $\MTL$-definable class of Kripke models and let $\varphi$ be an $\MTL$-formula that defines $\cC$. Let $k$ denote the modal depth of $\varphi$. We will show that the $\ML(\varovee)$-formula
\[
\varphi^* \dfn \Idis_{\Gamma \in \type_k(\cC)}  \big( \bigvee \Gamma  \,\big)
\]
defines $\cC$. Note that since $\type_k(\cC)$ is a family of sets of $k$-Hintikka formulas the outer disjunction is essentially finite. Likewise, since each $\Gamma$ is a collection of $k$-Hintikka formulas, it follows by flatness (remember that Hintikka formulas are $\ML$-formulas) that the inner disjunctions are essentially finite. Thus $\varphi^*$ is essentially a finite $\ML(\idis)$-formula.

Assume first that $\mK\in \cC$. By definition $\type_k(\mK)\in \type_k(\cC)$. Clearly, for each team $T$ of $\mK$, it holds that $\mK,T\models \bigvee \type_k(\mK)$, and thus that $\mK,T\models\varphi^*$. Therefore $\mK\models\varphi^*$. Assume then that $\mK\models \varphi^*$. Thus $\mK,W\models\varphi^*$, where $W$ is the domain of $\mK$. Therefore there exists a set $\Gamma\in \type_k(\cC)$ such that $\mK,W\models \bigvee\Gamma$. Thus $\type_k(\mK)= \type_k(\mK,W)\subseteq \Gamma$. Recall that $\cC=\modlss(\varphi)$.  Now since $\Gamma\in\type_k(\cC)=\type_k\big(\modlss(\varphi)\big)$, it follows from Lemma \ref{typelemma} that $\mK\in\modlss (\varphi)=\cC$.
\end{proof}
The following theorem directly follows.
\begin{theorem}\label{MTLisMDL}
A class $\cC$ of Kripke models is definable in $\MTL$ if and only if it is definable in $\ML(\varovee)$
\end{theorem}

Now by Propositions \ref{prop:ml<mdl} and \ref{prop:mdl<emdl}, by Theorems \ref{MINCisML} and \ref{MTLisMDL}, and the fact that $\EMDL\meq\ML(\idis)$  \cite{HeLuSaVi14}, we obtain the following trichotomy.

\begin{theorem}\label{thm:morder}
\(
\{ \ML, \MINC, \EMINC \} \mle \MDL \mle \{\EMDL, \ML(\varovee), \MTL \}.
\)
\end{theorem}

\subsection{Definability with respect to frames}
We now shift from model definability to frame definability. The objective of this section is to establish that the trichotomy of model definability  (Theorem \ref{thm:morder}) can be strengthened to the following dichotomy of frame definability.
\[
\{ \ML, \MINC, \EMINC \} \fle \{\MDL, \EMDL, \ML(\varovee), \MTL \}.
\]
It is easy to show that equality with respect to $\meq$ implies equality with respect to $\feq$:
\begin{lemma}\label{eqlemma}
Let $\LL$ and $\LL'$ be logics such that $\LL\meq\LL'$. Then  $\LL\feq\LL'$.
\end{lemma}
\begin{proof}
By symmetry it suffices to show that $\LL\fleq\LL'$.
Let $\fF$ be a Kripke frame, $\varphi$ an $\LL$-formula and $\varphi^*$ the related $\LL'$-formula such that $\modlss(\varphi)=\modlss(\varphi^*)$.  Now, by definition, $\fF \models \varphi$ if and only if $(\fF,V) \models \varphi$ for every valuation $V$. Since $\modlss(\varphi)=\modlss(\varphi^*)$, this holds if and only if  $(\fF,V) \models \varphi^*$ for every valuation $V$, which by definition holds if and only if $\fF \models \varphi^*$.
Now let $\cF$ be some $\LL$-definable class of Kripke frames and let $\Gamma$ be a set of $\LL$-formulas that defines $\cF$. Define $\Gamma^*\dfn \{\varphi^* \mid \varphi\in\Gamma\}$. Clearly $\Gamma^*$ is a set of $\LL'$-formulas that defines $\cF$.
\end{proof}
The only thing that is left to show is that with respect to frame definability $\MDL$ and $\EMDL$ conicide.
\begin{proposition}\label{mdlemdl}
Let $\Phi$ be an infinite set of proposition symbols. For every formula $\varphi\in \EMDL(\Phi)$ there exists a formula $\varphi^*\in\MDL(\Phi)$ such that $\fF\models\varphi$ iff $\fF\models\varphi^*$ for every frame $\fF$. 
\end{proposition}
\begin{proof}
We give a sketch of the proof here. A similar proof is given in \cite[Proposition 5.8]{virtema14}. The translation $\varphi \mapsto \varphi^*$ is defined inductively in the following way. For (negated) proposition symbols the translation is the identity. For propositional connectives and modalities we define
\[
(\psi_1 \oplus \psi_2) \mapsto (\psi_1^* \oplus \psi_2^*), \quad\text{and}\quad \nabla \psi \mapsto \nabla \psi^*,
\]
where $\oplus\in\{\wedge,\vee\}$ and $\nabla\in \{\Diamond, \Box\}$. The only nontrivial case is the case for the dependence atoms. Let $\varphi$ be the dependence atom $\dep{\psi_1,\dots,\psi_n}$, let $k$ be the modal depth of $\varphi$, and let $p_1,\dots,p_n$ be distinct fresh proposition symbols. Define
\[
\varphi^* \dfn \big(\bigwedge_{0\leq i\leq k} \Box^i \bigwedge_{1\leq j\leq n} (p_j\leftrightarrow \psi_j) \big) \rightarrow  \dep{p_0,\dots,p_n}. 
\]
It is now straightforward to show that the claim follows.
\end{proof}

Now from Theorem \ref{thm:morder}, Lemma \ref{eqlemma}, and Proposition \ref{mdlemdl}, we obtain the desired dichotomy.
\begin{theorem}\label{F:order}
$\{ \ML, \MINC, \EMINC \} \fle \{\MDL, \EMDL, \ML(\varovee), \MTL \}$. 
\end{theorem}

\section{Connecting team semantics and universal modality}\label{sec:mlandum}

Recall the model theoretic characterisations of model definability and frame definability for $\ML$, i.e.,  Theorems \ref{MLmodels}, \ref{gbth}, and \ref{fingbth}. By Theorems \ref{thm:morder} and \ref{F:order}, we directly obtain the corresponding characterisations for \MINC and $\EMINC$. Now recall the corresponding characterisations for $\ML(\uBoxp)$, i.e.,  Theorems \ref{MLupmodels}, \ref{GbTh4mlup}, and \ref{thm:GbTh4mlup}. In this section we will show that with respect to model and frame definability $\ML(\uBoxp)$ and $\ML(\varovee)$ coincide. Thus we obtain model theoretic characterisations of model definability for $\EMDL$, $\ML(\varovee)$ and $\MTL$. Similarly we obtain model theoretic characterisations of frame definability for $\MDL$, $\EMDL$, $\ML(\varovee)$ and $\MTL$. See Tables \ref{table:models}-\ref{table:finframe} for an overview of the characterisations. 

\begin{table}\centering
\scalebox{0.9}{
\begin{tabular}{cccccc}\toprule
Logic & \multicolumn{3}{c}{Closure under} & Elementary &References \\ \cmidrule{2-4}
&disjoint& surjective & total surjective  & &\\
&unions& bisimulations & bisimulations &  &\\\midrule
\ML & X & X & X  & X  & \cite{derijke01} \\
\MINC &  &  &  & & Thm. \ref{MINCisML}  \\
\EMINC &  &  &  & & Thm. \ref{MINCisML} \\ \cmidrule{1-1} \cmidrule{6-6}
$\ML(\uBoxp)$ &  & X & X & X  &  Cor. \ref{MLupmodelsb} \\
$\ML(\varovee)$ &  &  &  &  & Thm. \ref{midis is uboxp} \\
\EMDL &  &  &  & &  Thm. \ref{thm:morder}\\
\MTL &  &  &  &  & Thm. \ref{thm:morder} \\\cmidrule{1-1} \cmidrule{6-6}
$\ML(\uBox)$ &  &  & X & X & \cite{perkov12}
\\\bottomrule\\
\end{tabular}
}
\caption{Characterisation of model definability of different modal logics. Note that total surjectice bisimulations are special cases of surjective bisimulations. E.g., a class $\cC$ of Kripke models is definable in $\EMDL$  if and only if  $\cC$ is elementary, and closed under surjective bisimulations.}
\label{table:models}
\end{table}

\begin{table}\centering
\scalebox{0.9}{
\begin{tabular}{cccccccc}\toprule
Logic & \multicolumn{3}{c}{Closure under} &  & \multicolumn{2}{c}{Reflects} & References \\ \cmidrule{2-4}  \cmidrule{6-7}
&disjoint& bounded mor- & generated  && ultrafilter & finitely gene- & \\
&unions& phic images & subframes && extensions & rated subframes & \\\midrule
\ML & X & X & X &  & X & X\footnote{} & \cite{GT1975} \\
\MINC &  &  &  &  &  & & Thm. \ref{F:order}  \\
\EMINC &  &  &  &  &  & & Thm. \ref{F:order} \\ \cmidrule{1-1} \cmidrule{8-8}
$\ML(\uBoxp)$ &  &  &  &  &  & & Thm. \ref{GbTh4mlup} \\
$\ML(\varovee)$ &  &  &  &  &  & &Thm. \ref{idis is uboxp} \\
\MDL &  & X & X &  & X & X & Thm. \ref{F:order} \\
\EMDL &  &  &  &  &  &  &  Thm. \ref{F:order} \\
\MTL &  &  &  &  &  &  & Thm. \ref{F:order} \\\cmidrule{1-1} \cmidrule{8-8}
$\ML(\uBox)$ &  & X &  &  & X & & \cite[Cor. 3.9]{Goranko92}
\\\bottomrule\\
\end{tabular}
}
\caption{Characterisation of frame definability of different modal logics with respect to first-order definable frame classes. E.g., an elementary class $\cF$ of Kripke frames is definable in $\EMDL$ if and only if  $\cF$ is closed under taking generated subframes and bounded morphic images, and reflects ultrafilter extensions and finitely generated subframes.}
\label{table:frame1}
\end{table}
\footnotetext{If a class of frames is closed under disjoint unions and bounded morphic images then it reflects finitely generated subframes.}

\begin{table}\centering
\scalebox{0.9}{
\begin{tabular}{ccccc}\toprule
Logic & \multicolumn{3}{c}{Closure under} &  References \\ \cmidrule{2-4}
&disjoint& bounded morphic & generated  & \\
&unions& images & subframes & \\\midrule
\ML & X & X & X  &  \cite{Benthem1993} \\
\MINC &  &  &  & \Cref{F:order}  \\
\EMINC &  &  &  & \Cref{F:order} \\ \cmidrule{1-1} \cmidrule{5-5}
$\ML(\uBoxp)$ &  &  &  &  \Cref{thm:GbTh4mlup} \\
$\ML(\varovee)$ &  &  &  &  \Cref{idis is uboxp} \\
\MDL &  & X & X &  \Cref{F:order}\\
\EMDL &  &  &  &    \Cref{F:order} \\
\MTL &  &  &  &  \Cref{F:order} \\\cmidrule{1-1} \cmidrule{5-5}
$\ML(\uBox)\footnote{}$ &  & X &  & \cite{GG1993}
\\\bottomrule\\
\end{tabular}
}
\caption{Characterisation of relative frame definability of different modal logics within the class of finite transitive frames. E.g., a class $\cF$ of finite transitive frames is definable in $\EMDL$ within the class of finite transitive frames if and only if  $\cF$ is closed under taking generated subframes and bounded morphic images.}
\label{table:finframe}
\end{table}
\footnotetext{The characterisation for $\ML(\uBox)$ holds already within the class of finite frames}

We start with a normal form for $\ML(\varovee)$.
\begin{definition}
We say that an $\ML(\idis)$-formula $\varphi $ is in \emph{$\idis$-normal form} if
\(
\varphi=\psi_1\idis\psi_2\idis\dots\idis\psi_n
\)
for some $n\in\omega$ and $\psi_1,\psi_2,\dots,\psi_n\in \ML(\Phi)$.
\end{definition}

\begin{proposition}[$\idis$-normal form, \cite{virtema14, fanthesis}]
\label{disnormalform}
For every $\ML(\idis)$-formula $\varphi$ there exists an equivalent formula in $\idis$-normal form.
\end{proposition}

\begin{lemma}
\label{kvalid_eq_tvalid}
For every $\ML$-formula $\varphi$ and model $\mK$: $\mK \Vdash \uBox \varphi$ iff $\mK,W \models \varphi$.
\end{lemma}

\begin{proof}
By the semantics of $\uBox$, $\mK \Vdash \uBox \varphi$ iff $\mK, w \Vdash \varphi$ for every $w \in W$. Furthermore by Proposition \ref{mlextends}, $\mK, w \Vdash \varphi$ for every $w \in W$ iff $\mK, W \models \varphi$.
\end{proof}

\begin{lemma}
\label{vtoa}
For every $\ML(\idis)$-formula $\varphi$ there exists a formula $\varphi^-\in \bigvee \uBox\ML$ such that $\mK\models \varphi$ iff $\mK\Vdash \varphi^-$ for every Kripke model $\mK$. 
\end{lemma}

\begin{proof}
Let $\varphi$ be an arbitrary $\ML(\idis)$-formula. By Proposition \ref{disnormalform}, we may assume that
\(
\varphi = \psi_1\idis\cdots\idis\psi_n,
\)
for some $n\in \omega$ and $\psi_1,\dots,\psi_n\in \ML$. Let $\mK=(W,R,V)$ be an arbitrary model. It suffices to show $\mK\models \varphi \,\Leftrightarrow\,  \mK\Vdash \uBox\psi_1\ \vee\cdots\vee\uBox\psi_n$. 
This is shown as follows. 
\[
\begin{array}{rcl}
\mK\models \varphi &\quad\stackrel{\stackrel{\small\text{Def. of $\models$}}{\scriptsize \text{Proposition \ref{closures}}}}{\Leftrightarrow}\quad& \mK,W\models \psi_1\idis\cdots\idis\psi_n\\
&\quad\stackrel{\scriptsize \text{Def. of $\idis$}}{\Leftrightarrow}\quad& \text{There exists $i\leq n$: } \mK,W\models \psi_i \\
&\quad\stackrel{\scriptsize \text{Lemma \ref{kvalid_eq_tvalid}}}{\Leftrightarrow}\quad& \text{There exists $i\leq n$: } \mK \Vdash \uBox\psi_i \\
&\quad\stackrel{\scriptsize \text{Defs. of $\Vdash$, $\uBox$ and $\vee$}}{\Leftrightarrow}\quad& \mK\Vdash \uBox\psi_1\vee\cdots\vee\uBox\psi_n. 
\end{array}
\]
\end{proof}

\begin{lemma}
\label{elementarytoidis}
For every $\varphi\in\bigvee \uBox\ML$ there exists an $\ML(\idis)$-formula $\varphi^{*}$ such that $\mK\Vdash \varphi$ iff $\mK\models \varphi^*$ for every Kripke model $\mK$.
\end{lemma}

\begin{proof}
Let $\varphi\in\bigvee \uBox\ML$ be an arbitrary formula, i.e., $\varphi = \uBox\psi_1 \vee \dots \vee \uBox\psi_n$ for some $n\in \omega$ and $\psi_1,\dots,\psi_n\in \ML$. 
Let $\mK=(W,R,V)$ be an arbitrary Kripke model. It suffices to show $\mK\Vdash \varphi$ $\Leftrightarrow$ $\mK\models \psi_1 \idis \cdots\idis\psi_n$. 
We proceed as follows. 
\[
\begin{array}{rcl}
\mK\Vdash \uBox\psi_1 \vee \cdots \vee \uBox\psi_n &\quad\stackrel{\scriptsize \text{Defs. of $\Vdash$, $\uBox$, and $\vee$}}{\Leftrightarrow}\quad &\text{There exists $i \leq n$: } \mK \Vdash \uBox\psi_i\\
%
%
&\quad\stackrel{\scriptsize \text{Lemma \ref{kvalid_eq_tvalid}}}{\Leftrightarrow}\quad &\text{There exists $i\leq n$: } \mK,W\models \psi_i.\\
&\quad\stackrel{\scriptsize \text{Def. of $\idis$}}{\Leftrightarrow}\quad &\mK,W\models \psi_1 \idis\cdots\idis\psi_n\\
&\quad\stackrel{\scriptsize \text{Proposition \ref{closures}}}{\Leftrightarrow}\quad &\mK\models \psi_1 \idis\cdots\idis\psi_n.
\end{array}
\]
\end{proof}

\begin{theorem}\label{midis is uboxp}
A class $\cC$ of Kripke models 
is $\ML(\idis)$-definable iff it is $\ML(\uBoxp)$-definable. 
\end{theorem}
\begin{proof}
Let $\cC$ be a class of Kripke models. 
By Proposition \ref{prop:mlup_eq_vuml_m}, it suffices to show that $\cC$ is $\ML(\idis)$-definable iff it is $\bigvee\uBox\ML$-definable. ``If'' and ``Only If'' parts follow directly from Lemma \ref{vtoa} and Lemma \ref{elementarytoidis}, respectively.
\end{proof}

The following theorem then directly follows via Lemma \ref{eqlemma}.
\begin{theorem}
\label{idis is uboxp}
A frame class $\cF$ 
is $\ML(\idis)$-definable iff it is $\ML(\uBoxp)$-definable. 
\end{theorem}

We are finally ready to combine our results concerning model and frame definability of team-based modal logics and modal logics with the universal modality.  
By Propositions \ref{ummorder} and \ref{prop:frame_def_ordering}, and Theorems \ref{thm:morder}, \ref{F:order}, \ref{midis is uboxp}, and \ref{idis is uboxp}, we obtain the following strict hierarchies.
\begin{theorem}
\label{thm:comp_frame_def}
With respect to model and frame definability, we have the following hierarchies:
\begin{align*}
&\{ \ML, \MINC, \EMINC \} \mle \MDL \mle  \{\EMDL, \ML(\varovee),  \ML(\uBox{}^+), \MTL \} \mle  \ML(\uBox)\\
&\{ \ML, \MINC, \EMINC \} \fle \{\MDL, \EMDL, \ML(\varovee), \ML(\uBox{}^+), \MTL \} \fle  \ML(\uBox).
\end{align*}
\end{theorem}

We can now extend the characterisations of model and frame definability (i.e., Theorems  \ref{MLmodels}, \ref{MLupmodels}, \ref{gbth}, \ref{GbTh4mlup}, \ref{fingbth}, and \ref{thm:GbTh4mlup}) to cover also team-based logics.

\begin{corollary}
For every logic $\LL\in\{\ML,\MINC,\EMINC\}$ and every class $\cC$ of Kripke models, the following are equivalent:
\begin{itemize}
\item[$(\mathrm{i})$] $\cC$ is $\LL$-definable.
\item[$(\mathrm{ii})$] $\cC$ is closed under surjective bisimulations, disjoint unions and ultraproducts, and $\overline{\cC}$ is closed under ultrapowers.
\item[$(\mathrm{iii})$] $\cC$ is elementary and closed under surjective bisimulations and disjoint unions.
\end{itemize}
\end{corollary}

\begin{corollary}
For every logic $\LL\in\{\ML(\uBoxp),\EMDL, \ML(\varovee), \MTL\}$ and every class $\cC$ of Kripke models, the following are equivalent:
\begin{itemize}
\item[$(\mathrm{i})$] $\cC$ is $\LL$-definable.
\item[$(\mathrm{ii})$] $\cC$ is closed under surjective bisimulations and ultraproducts, and $\overline{\mathbb{C}}$ is closed under ultrapowers.
\item[$(\mathrm{iii})$] $\cC$ is elementary and closed under surjective bisimulations.
\end{itemize}
\end{corollary}

\begin{corollary}
For every logic $\LL\in\{\ML,\MINC,\EMINC\}$ and for every elementary frame class $\cF$, the following are equivalent:
\begin{itemize}
\item[$(\mathrm{i})$] $\cF$ is $\LL$-definable.
\item[$(\mathrm{ii})$] $\cF$  is closed under taking bounded morphic images, generated subframes, disjoint unions and reflects ultrafilter extensions.
\end{itemize}
\end{corollary}

\begin{corollary}
For every logic $\LL\in\{\ML(\uBoxp),\MDL,\EMDL,\ML(\idis)\}$ and for every elementary frame class $\cF$, the following are equivalent:
\begin{itemize}
\item[$(\mathrm{i})$] $\cF$ is $\LL$-definable.
\item[$(\mathrm{ii})$] $\cF$ is closed under taking generated subframes and bounded morphic images, and reflects ultrafilter extensions and finitely generated subframes. 
\end{itemize}
\end{corollary}

\begin{corollary}
For every logic $\LL\in\{\ML,\MINC,\EMINC\}$ and every class $\cF$ of finite transitive frames, the following are equivalent:
\begin{itemize}
\item[$(\mathrm{i})$] $\cF$ is $\LL$-definable within $\cF_{\mathrm{fintra}}$.
\item[$(\mathrm{ii})$] $\cF$  is closed under taking bounded morphic images, generated subframes, and disjoint unions.
\end{itemize}
\end{corollary}

\begin{corollary}
For every logic $\LL\in\{\ML(\uBoxp),\MDL,\EMDL,\ML(\idis)\}$ and every class $\cF$ of finite transitive frames, the following are equivalent:
\begin{itemize}
\item[$(\mathrm{i})$] $\cF$ is $\LL$-definable within $\cF_{\mathrm{fintra}}$.
\item[$(\mathrm{ii})$] $\cF$ is closed under taking generated subframes and bounded morphic images.
\end{itemize}
\end{corollary}

\section{Conclusion}
In this paper we studied model and frame definability of different modal logics. The first half of this article considered the extension of modal logic with the universal modality whereas the second half concetrated in modal logics with team semantics. We showed that with respect to model and frame definability a strict linear hierarchy between all of the logics studied here emerge, respectively. With respect to model definability we have four distinct cases, whereas in frame definability only three remain, see Theorem \ref{thm:comp_frame_def} for the hierarchies. Moreover, we gave model theoretic characterisations for model definabity (see Table \ref{table:models}) and frame definability; restricted to elementary classes and to the class of finite transitive frames (see Tables  \ref{table:frame1} and \ref{table:finframe}, respectively).

Note that our results imply that with respect to model definability every logic between $\EMDL$ and $\MTL$ coincide. Similarly, with respect to frame definability, every logic between $\MDL$ and $\MTL$ coincide. In particular, we obtain results concerning \emph{modal independence logic} $\MIL$ and \emph{extended modal independence logic} $\EMIL$ (for definitions see \cite{KMSV14}), since with respect to expressive power $\MDL\le \MIL\le\MTL$ and  $\EMDL\le \EMIL\le\MTL$.

We conclude with some open questions:
\begin{itemize}
\item Where does $\MIL$ lie with respect to model definability?
\item Is there some natural fragment of $\ML(\uBoxp)$ that coincides with $\MDL$ or $\MIL$ with respect to model definability?
\item Can we give model theoretic characterisations for model definability of $\MDL$ and $\MIL$?
\item  Can we use the notion of local bounded morphism (cf. \cite{Benthem1988}) to drop the requirement of transitivity from \Cref{thm:GbTh4mlup}?
\end{itemize}

\section*{References}

\bibliographystyle{plain}
\bibliography{tampere}

\end{document}